\documentclass{amsart}
\numberwithin{equation}{section}
\usepackage{graphicx}
\usepackage{amssymb}
\usepackage{amsmath}
\usepackage{setspace}
\usepackage{color}
\allowdisplaybreaks

\usepackage{hyperref}
\usepackage{cleveref}

\usepackage{xcolor}
\usepackage{graphicx} % Required for inserting images
\usepackage{xfrac}
\usepackage{amsfonts, bm}
\newcommand{\Aut}{{\rm Aut}}
\newcommand{\Homeo}{{\rm Homeo}}
\newcommand{\Sub}{{\rm Sub}}
\newcommand{\mbb}{\mathbb}
\newcommand{\N}{{\mathbb N}}
\newcommand{\Z}{{\mathbb Z}}
\newcommand{\R}{{\mathbb R}}
\newcommand{\C}{{\mathbb C}}

\newcommand{\Sc}{{\mathbb S}}
\newcommand{\gf}{{\mathfrak g}}

\newcommand{\lf}{{\mathfrak l}}

\newcommand{\Hc}{{\mathcal H}}

\newcommand{\K}{{\mathcal K}}
\newcommand{\D}{{\mathcal D}}
\newcommand{\U}{{\mathcal U}}
\newcommand{\cS}{{\mathcal S}}
\newcommand{\Bc}{{\mathcal B}}
\newcommand{\RP}{{\mathbb{RP}}}

\newcommand{\spg}{{\rm Sub}^p_G}
\newcommand{\spp}{{\rm Sub}^{p}}
\newcommand{\du}{{\rm d\,}}
\newcommand{\Id}{{\rm Id}}
\newcommand{\Inn}{{\rm Inn}}
\newcommand{\inn}{{\rm inn}}
\newcommand{\GL}{{\rm GL}}
\newcommand{\SL}{{\rm SL}}
\newcommand{\T}{{\overline{T}}}

\let\ol=\overline
\let\ap=\alpha

\let\mi=\setminus
\newtheorem{thm}{Theorem}[section]
\newtheorem{cor}[thm]{Corollary}
\newtheorem{lem}[thm]{Lemma}
\newtheorem{prop}[thm]{Proposition}

\newtheorem{defn}[thm]{Definition}
\newtheorem{rem}[thm]{Remark}

\newtheorem{example}[thm]{Example}
\theoremstyle{remark}

\newcommand{\cp}{{\mathcal{C}'}}
\newcommand{\cP}{{\mathcal{P}}}

\begin{document}

\title[Distal actions on space of one-parameter subgroups]{Characterisation of distal 
actions of automorphisms on the space of one-parameter subgroups of Lie groups}

\author{Debamita Chatterjee}
\address{Debamita Chatterjee\\
School of Physical Sciences\\
Jawaharlal Nehru University\\
New Delhi 110 067}

\email{debamita.math@gmail.com}

\author{Riddhi Shah}
\address{Riddhi Shah\\
School of Physical Sciences\\
Jawaharlal Nehru University\\
New Delhi 110 067}

\email{rshah@jnu.ac.in, riddhi.kausti@gmail.com}
\date{February 15, 2025}

\begin{abstract}
For a connected Lie group $G$ and an automorphism $T$ of $G$, we consider the action of $T$ on Sub$_G$, the  
compact space of closed subgroups of $G$ endowed with the Chabauty topology. We study the 
action of $T$ on Sub$^p_G$, the  closure in Sub$_G$ of the set of closed one-parameter subgroups of $G$. 
We relate the distality of the $T$-action on Sub$^p_G$ with that of the $T$-action on $G$ and characterise 
the same in terms of compactness of the closed subgroup generated by $T$ in Aut$(G)$ when $T$ acts distally 
on the maximal central torus and $G$ is not a vector group. We extend these results to the action of a subgroup
of Aut$(G)$, and equate the distal action of any closed subgroup ${\mathcal H}$ on Sub$^p_G$ 
with that of every element in ${\mathcal H}$. Moreover, we show that a connected Lie group $G$ acts distally 
on Sub$^p_G$ by conjugation if and only if $G$ is either compact or it is isomorphic to a direct product 
of a compact group and a vector group. Some of our results extend those of Shah and Yadav. 
\end{abstract}

\maketitle

{\small 

}

\tableofcontents

\section{Introduction}

For a Hausdorff topological space $X$, a homeomorphism $T$ of $X$  
is said to be {\it distal} (equivalently, $T$ acts {\it distally} on $X$) if for any pair of distinct 
elements $x,y \in X$, the closure of the double orbit 
$\{(T^n(x), T^n(y)) \mid n\in \Z\}$ in $X \times X$ does not intersect 
the diagonal, i.e.\ for $x,y \in X$ with $x \neq y $, 
$\ol{\{(T^n(x), T^n(y))\mid n \in \Z\}} \cap \{(a,a)\mid a \in X\}=\emptyset$. 
David Hilbert introduced the notion of distality to study non-ergodic actions on compact spaces 
(see Ellis \cite{E} and Moore \cite{M9}). Distal actions on compact spaces, Lie groups as well as 
on locally compact groups have been studied by many mathematicians in different contexts 
(see Abels \cite{Ab1, Ab2}, Choudhuri et al \cite{CFY}, Ellis \cite{E}, 
Furstenberg \cite{Fu}, Palit-Shah \cite{PaS}, Palit-Prajapati-Shah \cite{PaPrS}, Raja-Shah 
\cite{RS1, RS2}, Shah \cite{Sh1}, Shah-Yadav \cite{SY1, SY2, SY3}, and references cited therein).

For a locally compact (Hausdorff) group  $G$ with the identity $e$, let $\Sub_G$ 
denote the set of all closed subgroups of $G$  endowed with the Chabauty topology. 
Recall that $\Sub_G$ is compact and Hausdorff, and it is metrisable if $G$ is 
second countable \cite{BP}. Let $\Aut(G)$ denote the group of all automorphisms of $G$; 
an automorphism is both a homomorphism and a homeomorphism. 
There is a canonical group action of $\Aut(G)$ on $\Sub_G$; namely, $(T,H)\mapsto T(H)$, $T \in \Aut(G)$, 
$H\in \Sub_G$; this action gives a homomorphism from $\Aut(G)$ to $\Homeo(\Sub_G)$. As the image
of $\Aut(G)$ under this action is a large subclass of homeomorphisms of $\Sub_G$, it is  
important to study the dynamics of the action of this special subclass in terms of distality. 

In \cite{SY3}, Shah and Yadav have studied the action of automorphisms of a connected Lie group $G$ 
on $\Sub_G$ extensively and have shown in particular that for a large class of $G$, which does not have 
any compact central subgroup of positive dimension, an automorphism $T$ of $G$ acts distally on 
$\Sub^a_G$, the space of closed abelian subgroups of $G$, if and only if $T$ generates 
a relatively compact subgroup in $\Aut(G)$. A similar characterisation also holds for $T$ 
belonging to the class (NC), i.e.\ those 
$T$ for which the closure of the $T$-orbit of any discrete cyclic group does not contain 
the trivial group $\{e\}$. Note that $\Sub^a_G$ is very large for many groups $G$; e.g.\ it is the 
same as $\Sub_G$ if $G$ is abelian. It also contains all discrete cyclic subgroups. 
Shah, with Palit in \cite{PaS}, and with Palit and Prajapati in \cite{PaPrS}, has 
studied distal actions of automorphisms of a discrete group $G$ on $\Sub^a_G$, where $G$ is  
either polycyclic or a lattice in a connected Lie group. 

The question arises whether one can characterise distal actions of automorphisms on other smaller invariant 
subspaces of $\Sub^a_G$ for a connected Lie group $G$. Most of the techniques in earlier result mentioned 
above involve study of the actions 
on a class of discrete abelian subgroups. Here, we consider the class of smallest closed connected 
abelian subgroups; namely, closed one-parameter subgroups of $G$ and study the dynamics of distal actions of 
automorphisms on it. Let $\spg$ denote the smallest closed subset of $\Sub_G$ containing all closed 
one-parameter subgroups of $G$. Note that $\spg$ is compact and it is invariant under the action of $\Aut(G)$. 
We get a characterisation for a large class of $G$ as follows. 

 \begin{thm}\label{main}
 Let $G$ be a connected Lie group without any non-trivial compact connected central subgroup and 
 let $T$ be an automorphism of $G$. Then the following hold:
 \begin{enumerate} 
 \item[{(1)}] If $G$ is abelian, then $G$ is isomorphic to $\R^d$ for some $d\in\N$, and  
 $T$ acts distally on $\spg$ if and only if $T \in \K\D$, where $\K$ is a 
 compact subgroup of $\GL(d,\R)$  and $\D$ is the center of $\GL(d,\R)$.
 \item[{(2)}] If $G$ is not abelian, then $T$ acts distally on $\spg$ if and only if $T \in \K$, a compact 
 subgroup of $\Aut(G)$.
 \end{enumerate}
 \end{thm} 

Note that the condition on the center in Theorem \ref{main} holds if $G$ is simply connected, 
or more generally its nilradical is simply connected. 
It also holds for $G/M$ for any connected Lie group $G$, where $M$ is the maximal central torus in $G$ 
(the maximal compact connected central subgroup in $G$). The result in the case of non-abelian 
groups $G$ in (2) as above gives a similar characterisation as in Theorem 4.1 of \cite{SY3}, even though 
the space $\spg$ is much smaller than $\Sub^a_G$. In case $G$ as above is abelian, then $G=\R^d$, 
a vector group, and $\spg$ is homeomorphic to $\RP^{d-1}\sqcup\{\{0\}\}$, and the 
Theorem \ref{main}\,(1) shows that a larger class of $T$ acts distally on $\spg$.

In \cite{SY3}, it is shown for a connected Lie group $G$ without any compact central subgroup 
of positive dimension that if $T$ acts distally on $\Sub^a_G$, then 
it acts distally on $G$ (more generally, see Theorem 3.6 and Corollary 3.7 of \cite{SY3}). We generalise 
this for Lie groups which are not vector groups; as in the later case, Theorem \ref{main}\,(1) shows that it is 
possible to have automorphisms which act distally on $\spg$ but not on $G$.

\begin{thm} \label{distal-g}
Let $G$ be a connected Lie group which is not a vector group and let $T\in\Aut(G)$ 
be such that it acts distally on the largest compact connected central subgroup of $G$. 
If $T$ acts distally on $\spg$, then $T$ acts distally on $G$. 
\end{thm}

An automorphism $T$ of a connected Lie group $G$ is said to be unipotent if 
all the eigenvalues of $\du T$ are equal to 1, where $\du T$ is the corresponding 
Lie algebra automorphism of the Lie algebra of $G$. 
It is known that $T$ acts distally on $G$ if and only if all the eigenvalues of $\du T$ have absolute value 1 
 (cf.\ \cite{Ab1}). Any unipotent automorphism acts distally on $G$ and if it belongs to a compact subgroup 
 of $\Aut(G)$, then it is trivial. Theorem 4.3 of \cite{SY3} shows that no non-trivial unipotent automorphism 
 acts distally on $\Sub^a_G$; generalising this statement we prove the following.
 
\begin{thm} \label{unip} Let $G$ be a connected Lie group and let $T\in\Aut(G)$ be unipotent. 
Then $T$ acts distally on $\spg$ if and only if $T=\Id$, the identity map. 
\end{thm}

In particular, Theorem \ref{unip} illustrates that the converse of Theorem \ref{distal-g} does not 
hold as the non-trivial unipotent automorphisms do not act distally on $\spg$.

Let $\Sub^c_G$ denote the set of all discrete cyclic subgroups of $G$. It is invariant under 
the action of $\Aut(G)$ and so is its closure in $\Sub_G$. In \cite{PaS} and \cite{PaPrS}, distal 
actions of automorphisms of  $\Sub^c_\Gamma$ have been characterised, where $\Gamma$ is 
a certain type of discrete group, e.g.\ a discrete polycyclic group or a lattice in a connected Lie group. 
In a connected Lie group $G$, any closed one-parameter subgroup $H=\{x_t\}_{t\in\R}$ is the limit of 
$\{H_n\}_{n\in\N}$ as $n\to\infty$ in $\Sub_G$, where $H_n=\{x_{m/n}\mid m\in\Z\}$, $n\in\N$, are 
discrete cyclic groups. 
Thus $\spg\subset \ol{\Sub^c_G}$ (the closure of $\Sub^c_G$ in $\Sub_G)$ and the latter is contained 
in $\Sub^a_G$. We should note that $\ol{\Sub^c_G}$ is much bigger than $\spg$ as the former contains 
all discrete cyclic subgroups (see Lemma \ref{discrete}). If $T$ acts distally on $\ol{\Sub^c_G}$, then $T\in$ (NC). 
It is easy to see that Theorems \ref{main}\,(2), \ref{distal-g} and \ref{unip} remain valid if $\spg$ is 
replaced by $\ol{\Sub^c_G}$ in the respective statements.

 The following generalises a part of Theorem 4.1 of \cite{SY3} and also extends Theorem~\ref{main}\,(2).

\begin{thm} \label{general} Let $G$ be a connected Lie group such that it is not a vector group. Let 
$T\in\Aut(G)$ be such that it acts distally on the largest compact connected central subgroup of $G$. 
Then the following are equivalent:
\begin{enumerate}
\item $T$ acts distally on $\spg$.
\item $T$ acts distally on $\ol{\Sub^c_G}$. 
\item $T$ acts distally on $\Sub^a_G$. 
\item $T$ acts distally on $\Sub_G$. 
\item $T$ is contained in a compact subgroup of $\Aut(G)$. 
\end{enumerate}
\end{thm}

We may note that in case $G$ is a vector group, Theorem \ref{main}\,(1) shows that $(1)$ and $(5)$ above are not equivalent;
however, in this case, statements $(2-5)$ above are equivalent by Theorem 4.1 of \cite{SY3} as (2) above implies that 
$T\in ({\rm NC})$.

Recall that a group $\Gamma$ of homeomorphisms of a (Hausdorff) topological space $X$ acts distally on $X$ if 
$\ol{\{(\gamma(x),\gamma(y))\mid \gamma\in\Gamma\}}\cap \{(d,d)\mid\ d\in X\}=\emptyset$ for any $x,y\in X$ with $x\ne y$. 
In particular, a homeomorphism $T$ of $X$ acts distally on $X$ if and only if the cyclic group $\{T^n\mid n\in\Z\}$ acts distally on 
$X$. For a Lie group $G$, any compact subgroup of $\Aut(G)$ acts distally on $\Sub_G$ (see e.g.\ Lemma 2.2 in \cite{SY2}). 
It follows from Theorem 1.1 of \cite{Ab2} that for a subgroup $\Hc$ of $\Aut(G)$, 
$\Hc$ acts distally on $G$ if and only if $\Hc$ acts distally on the Lie algebra of $G$, 
and these statements are equivalent to the statement that every element of $\Hc$ acts distally on $G$ (see also \cite{CG}). 
Let $\ol{\Hc}$ denote the closure of $\Hc$ in $\Aut(G)$. It is easy to see using Abel's results in \cite{Ab1, Ab2} that 
$\Hc$ acts distally on $G$ if and only if so does $\ol{\Hc}$. It follows from Ellis' result in \cite{E} that 
$\Hc$ acts distally on any compact $\Hc$-invariant subspace of $\Sub_G$ if and only if so does $\ol{\Hc}$. 

 Using Theorems \ref{main} and 
\ref{general}, we get the following theorem for the action of any subgroup $\Hc$ of $\Aut(G)$. If $\Hc$ is closed, 
the theorem shows that $\Hc$ acts distally on $\spg$ if and only if so does every element of $\Hc$.

\begin{thm} \label{subgp-action} Let $G$ be a connected Lie group. Let $\Hc$ be a subgroup of $\Aut(G)$ 
such that it acts distally on the largest compact connected central subgroup of $G$. 
Consider the following statements:
\begin{enumerate}
\item Every element of $\ol{\Hc}$ acts distally on $\spg$. 
\item $\Hc$ acts distally on $\spg$.
\item $\Hc$ acts distally on $\ol{\Sub^c_G}$. 
\item $\Hc$ acts distally on $\Sub^a_G$. 
\item $\Hc$ acts distally on $\Sub_G$. 
\item $\ol{\Hc}$ is a compact group.
\end{enumerate}
\noindent Then, $(1)$ and $(2)$ are equivalent and $(3-6)$ are equivalent.\\ 
If $G$ is not a vector group, then $(1-6)$ are equivalent.\\
If $G$ is a vector group, i.e.\ if $G=\R^d$, for some $d\in\N$, then $(1)$ and $(2)$ are equivalent to the following:
\begin{enumerate}
\item[(7)] $\Hc\subset \K\D$, where $\K$ is a compact subgroup of $\GL(d,\R)$ and $\D$ is the center of $\GL(d,\R)$.
\end{enumerate}
\end{thm}

Note that in Theorem \ref{subgp-action}, one can also have equivalent statements similar to $(1)$ and $(2)$, where $\spg$ 
is replaced by any of $\ol{\Sub^c_G}$, $\Sub^a_G$ or $\Sub_G$. 
We may note here that one can not replace $\ol{\Hc}$ by $\Hc$ in (1) of Theorem \ref{subgp-action}, 
as there is a subgroup $\Hc$ of $\SL(4,\C)$ in which every element of $\Hc$ is contained in a compact subgroup 
of $\SL(4,\C)$ but the closure of $\Hc$ in $\SL(4,\C)$ is connected and contains unipotent elements (cf.\ \cite{Ba}). 
The condition that $T$ (resp.\ a subgroup $\Hc$ of $\Aut(G)$) acts distally on the maximal compact connected central 
subgroup (maximal central torus) of $G$ is satisfied easily if $T\in\Aut(G)^0$, 
(resp.\ $\Hc\subset \Aut(G)^0$), as $\Aut(G)^0$ acts trivially on any compact central subgroup. It is also satisfied if 
$T^n\in\Aut(G)^0$ for some $n\in\N$ (resp.\ $\Hc$ or $\Hc\Aut(G)^0$ has 
finitely many connected components). Here, $\Aut(G)^0$ denotes the connected component of the identity in $\Aut(G)$.

Let $\Inn(G)$ denote the group of inner automorphisms of $G$. It is a (not necessarily closed) normal 
subgroup of $\Aut(G)$. The following corollary generalises the first part of Corollary 4.5 in \cite{SY3}. 

\begin{cor} \label{inner}
Let $G$ be a connected Lie group. Then the following are equivalent: 
\begin{enumerate}
\item Every inner automorphism of $G$ acts distally on $\spg$.
\item Every inner automorphism of $G$ acts distally on $\Sub_G$. 
\item $\Inn(G)$ acts distally on $\spg$.
\item $\Inn(G)$ acts distally on $\Sub_G$. 
\item $G$ is compact or $G=\R^n\times K$, for some $n\in\N$ and the maximal compact normal 
subgroup $K$ of $G$. 
\end{enumerate}
\end{cor}

In Corollary \ref{inner}, one can also write equivalent statements involving the action on $\ol{\Sub^c_G}$. 

Note that $\spg$ is contained in a larger compact set $\Sub^{co}_G$, which is the closure of the set of all closed connected 
subgroups of $G$. Hence it is easy to show that Theorems \ref{main}--\ref{unip} are all valid if $\spg$ is replaced by
$\Sub^{co}_G$. In Theorems \ref{general} and \ref{subgp-action} and Corollary \ref{inner}, the action of a given 
automorphism or a subgroup of $\Aut(G)$ is distal on $\Sub^{co}_G$ if and only if it is so on $\spg$, under the condition 
on the group and the action on the central torus. 

Let us note some related known results. Ellis \cite{E} has shown that a semigroup $\Gamma$ of homeomorphisms of 
a compact space $X$ acts distally on $X$ 
if and only if the closure of $\Gamma$ in $X^X$ is a group. Note that minimal distal actions on compact metric 
spaces are characterised by Furstenberg \cite{Fu}. For many groups $G$, compact spaces $\Sub_G$, 
Sub$^a_G$ and $\Sub^c_G$ are identified (see Baik-Clavier \cite{BC1, BC2}, Bridson et al \cite{BHK}, Hamrouni and Kadri 
\cite{HamKa}, Kloeckner \cite{Kl} and also Pourezza and Habbard \cite{PH}); 
one can also identify $\spg$ for some groups $G$. The action of $\Aut(G)$ on some subspaces of $\Sub_G$ for the 3-dimensional 
Heisenberg group $G$ has also been described in detail \cite{BHK}. The space of $G$-invariant measures 
on Sub$_G$ and in particular, ergodic $G$-invariant measures on Sub$_G$ for some $G$ have also been studied extensively (see 
Gelander \cite{G} and the references cited therein). Moreover, the study of the action of $\Aut(G)$ on $\Sub_G$ and on 
its closed (compact) invariant subspaces leads to a better understanding of dynamics on these spaces. 

We now fix some notations. Let $G$ be a connected Lie group. Let $R$ (resp.\ $N$) denote the radical (resp.\ nilradical) of $G$, 
i.e.\ the largest connected solvable (resp.\ nilpotent) normal subgroup of $G$. Both these subgroup are characteristic in $G$. 
Recall that $G$ is said to be semisimple if its radical is trivial. In particular, $G/R$ is semisimple unless $G$ is solvable. 
For a closed subgroup $H$ of $G$, 
let $H^0$ denote the connected component of the identity $e$ in $H$, $[H,H]$ denote the commutator subgroup of $H$. Note that 
 $H^0$, $[H,H]$ and its closure $\ol{[H,H]}$ are characteristic in $H$. Let $\gf$ denote the Lie algebra of the connected Lie group $G$. 
 There is an exponential map $\exp:\gf\to G$ which is a continuous map and its restriction to a small open 
 neighbourhood of $0$ in $\gf$ is a homeomorphism onto an open neighbourhood of the identity $e$ in $G$. 
 Given $T\in\Aut(G),$ there exists a corresponding Lie algebra automorphism 
 $\du T : \gf\to \gf$ such that $\exp\circ\,\du T =T\circ\,\exp$.  The map 
 $\du: \Aut(G) \rightarrow \Aut(\gf)$, $T \mapsto \du T$ is an injective continuous 
 homomorphism and its image is closed, and in fact, $\du$ is a homeomorphism onto its image 
 (cf.\ \cite{Ch}, see also \cite{Hoc})
It is known that $\Aut(\gf)$ is a closed subgroup of $\GL(\gf)$, and we can view $\Aut(G)$ 
as a closed subgroup of $\GL(\gf)$. The topology on $\Aut(G)$, when considered 
as a subspace topology of $\GL(\gf)$, coincides with the (modified) compact-open topology (see \cite{St06}). 
We refer the reader to \cite{Hoc2} for the basic structure theory of Lie groups.

In \S 2, we study general properties of $\Sub_G$ and prove some elementary results about the structure of $\spg$. 
In \S 3, we discuss some properties and actions of automorphisms of a connected Lie group $G$ on $G$ as well as 
$\spg$ and discuss some preliminary results. In \S4 we prove the main results stated above. 

\section{Properties and structure of \texorpdfstring{${\mathrm{Sub}}_G$}{Sub G} and 
\texorpdfstring{${\mathrm{Sub}}^p_G$}{Sub p G}}

In this section, we discuss some basic properties of the Chabauty topology on the space $\Sub_G$. 
We also discuss some elementary results about the structure of  $\Sub_G$ and $\spg$.

Let $L$ be a locally compact Hausdorff topological group.
For a compact set $K$ and an open set $U$ in $L$, let 
$U_1(K)=\{ H \in \Sub_L \mid H \cap K = \emptyset \}$ and 
$U_2(U)= \{ H \in \Sub_L \mid H \cap U \neq \emptyset \}$. Here, 
$\{ U_1(K) \mid K \mbox { is compact} \} \cup \, \{ U_2(U) \mid U \mbox{ is open} \}$ 
is a sub-basis of the Chabauty topology on $\Sub_L$. As mentioned earlier, $\Sub_L$ is 
compact and Hausdorff, and if $L$ is second countable, then $\Sub_L$ is metrisable 
(cf.\  \cite{BP}, Lemma E.1.1). In particular $\Sub_L$ is metrisable if $L$ is a connected 
Lie group. For a closed subgroup $H$ of $L$, it is easy to see that $\Sub_H$ carries 
the subspace topology of $\Sub_L$. The following lemma is known (see e.g.\ Proposition E.1.2 in \cite{BP}) 

\begin{lem}\label{conv}
Let $G$ be a connected Lie group. A sequence 
$\{H_n\}\subset \Sub_G$ converges to $H\in \Sub_G$ 
if and only if the following hold:
\begin{enumerate}
\item[${\rm(I)}$] For $g\in G$, if there exists a subsequence $\{H_{n_k}\}$ of 
$\{H_n\}$ with $h_k\in H_{n_k}$, $k\in\N$, such that $h_k\to g$ in $G$, then $g\in H$. 
\item[${\rm(II)}$] For every $h\in H$, there exists a sequence $\{h_n\}_{n\in\N}$ such that
 $h_n\in H_n$, $n\in\N$, and $h_n\to h$.
\end{enumerate}
\end{lem}

The following lemma is a consequence of the well-known result that the set of discrete subgroups of a connected Lie group 
$G$ is open in $\Sub_G$. We will give a short proof for the sake of completeness.

\begin{lem} \label{discrete}
Let $G$ be a connected Lie group.  Then the trivial subgroup $\{e\}$ is isolated in $\Sub^{co}_G$, and so in $\spg$. Moreover, 
$\Sub^{co}_G$ does not contain any non-trivial discrete subgroup.
\end{lem} 

\begin{proof} Let $\mathfrak{C}$ be the set of closed connected subgroups of $G$. We have $\ol{\mathfrak{C}} = \Sub^{co}_G$. 
By Proposition E.1.5 in \cite{BP} (see also Proposition 3.4(iii) in \cite{BHK}, or also Corollary of Lemma 1.2 in
\cite{W}) the set $\mathfrak{D}$ of discrete subgroups of $G$ is open in $\Sub_G$ and so $\{\{e\}\} \subseteq\mathfrak{D}\cap \Sub^{co}_G
=\mathfrak{D} \cap \ol{\mathfrak{C}}\subseteq\ol{\mathfrak{D}\cap\mathfrak{C}} = \{\{e\}\}$. Hence the assertions hold.
\end{proof}.

Note that the action of $\Aut(G)$ on $\Sub_G$ is continuous for connected Lie groups $G$ (cf.\ \cite{SY3}, Lemma 2.4). 
In particular, if $T$ generates a relatively compact group in $\Aut(G)$, then it acts distally on $G$ 
as well as on $\Sub_G$ (see e.g.\ \cite{PaS}, Lemma 4.2 for the latter statement). 

A (continuous real) one-parameter subgroup in $G$ is the image of a continuous 
homomorphism from $\mathbb{R}$ to $G$. Any one-parameter subgroup is 
of the form $\{\exp tX\}_{t\in\R}$ for some $X$ in the Lie algebra $\gf$ of $G$. Moreover, if $H$ is a one-parameter 
subgroup in $G$, then so is $T(H)$ for any $T\in\Aut(G)$. Let 
$$
\cP_1(G):= \{ A \mid A  \mbox{ is a closed one-parameter subgroup of } G \}.$$ 
Now $\spg$ is the closure (in $\Sub_G$) of the set $\cP_1(G)$, i.e.\ 
$\spg : = \ol{\cP_1(G)}$, which is  compact and invariant under the action of $\Aut(G)$. 

Note that if any one-parameter subgroup in a connected Lie group is not closed, then its closure is compact 
and isomorphic to an $n$-dimensional torus $\mbb{T}^n=(\mbb{S}^1)^n$ for some $n\in\N$ with $n\geq 2$. 
If it is closed and bounded (compact), then it is isomorphic to 
$\mbb{S}^1$. If it is unbounded (not relatively compact), then it is closed and isomorphic to $\R$. 

For $\{x_t\}_{t\in\R}$, which may be a continuous one-parameter set or a one-parameter subgroup, we may just write 
$\{x_t\}$ for brevity, i.e.\ we may omit the parameter set if there is no ambiguity.

Now we prove some elementary results about the structure of $\spg$. The following lemma shows that if $G$ contains 
any non-trivial compact connected abelian subgroup of   
dimension at least 2, then $\cP_1(G)$ is not closed in $\Sub_G$. 

\begin{lem} \label{subpg}
Let $G$ be a connected Lie group. Then the following hold: 
\begin{enumerate}
\item All compact connected abelian subgroups of $G$ belong to $\spg$.
\item If $\{x_t\}\in\cP_1(G)\subset\spg$ and if $K$ is a compact connected abelian subgroup of $G$ 
which is centralised by each $x_t$. 
Then $\{x_t\}K\in\spg$.
\end{enumerate}
\end{lem}

\begin{proof}
Let $K$ be a compact connected abelian subgroup of $G$. Then the set $F$ of all finite order elements in $K$ is 
dense in $K$ and each such element of $F$ is contained in a closed one-parameter subgroup of $K$. 
Since $K$ is monothetic (i.e.\ it has a dense cyclic subgroup), we get that $K$ is a limit of a sequence of 
compact one-parameter subgroups of $K$ in $\Sub_K$. Therefore, $K\in\spp_K\subset\spg$. Thus (1) holds. 

Let $\{x_t\}\in\cP_1(G)\subset \spg$ be such that each $x_t$ centralises $K$. If $\{x_t\}$ is compact, then 
$\{x_t\}K$ is a compact connected abelian subgroup of $G$ and hence $\{x_t\}K\in\spg$ from (1). 
Now suppose $\{x_t\}$ is isomorphic to $\R$. Let $k\in K$ be such that it generates a 
dense cyclic subgroup in $K$. Since $K$ is exponential, there exists a one-parameter subgroup 
$\{k_t\}$, which is dense in $K$ and $k=k_1$. Let 
$A_n=\{a_t(n)=x_{t/n}k_t\}_{t\in\R}$. Then each $A_n$ is a closed one-parameter subgroup 
in $G$, and passing to a subsequence if necessary, we get that $A_n\to A$ for some subgroup $A\in\spg$. 
Note that $\{x_t\}\times K$ is a closed subgroup of $G$ and it contains each $A_n$. Therefore,  
$A\subset \{x_t\}\times K$. Let $t\in\R$ be fixed. Since $a_t(n)=x_{t/n}k_t\to k_t$ as $n\to\infty$, we get 
that $k_t\in A$, and hence $K\subset A$. Let $b_n=a_{tn}(n)=x_tk_{tn}\in A_n$, $n\in\N$. 
Passing to a further subsequence of $\{b_n\}$ if necessary, 
we get that $b_n\to x_th$ for some $h\in K$, and hence, $x_t\in A$, as $K\subset A$. 
Now $\{x_t\}\subset A$. Hence, $A=\{x_t\}\times K$ and (2) holds. 
\end{proof}

The following lemma relates the set $\spp_{G/K}$ and a compact subset of $\spg$ which is invariant 
under the action of $\Aut(G)$.

\begin{lem} \label{subpgk} Let $G$ be a connected Lie group, $K$ be a compact connected central subgroup of 
$G$ and let $\pi:G\to G/K$ be the natural projection. Then the following hold:
\begin{enumerate}
\item $\{HK\mid H\in\Sub_G\}$ is a closed \rm{(}compact\rm{)} subset of $\Sub_G$ and $\{HK\mid H\in\spg\}$ is a closed 
\rm{(}compact\rm{)} subset of $\spg$, and they are $T$-invariant if $T(K)=K$, where $T\in\Aut(G)$. 
\item $\cP_1(G/K)=\{\pi(A)\mid A\in\cP_1(G)\}$. 
\item $\Sub^p_{G/K}=\{\pi(H)\mid H\in\spg\}$.
\item $\Sub^p_{G/K}$ is homeomorphic to $\{HK\mid H\in\spg\}$, and the latter is the same set as $\{H\in\spg\mid HK=KH=H\}$. 
\end{enumerate}
\end{lem}

\begin{proof} Let $K$ and $\pi$ be as above. Note that $K\in\spg$ by Lemma \ref{subpg}. If $G=K$, then it is easy to see that 
$(1-4)$ are trivially satisfied. 

Now suppose that $G\ne K$. Since $K$ is compact and central in $G$, $HK$ is a closed subgroup of $G$ for any 
closed subgroup $H$ of $G$. If $H_nK\to L$ in $\Sub_G$, then since both $\Sub_G$ and $K$ are compact, it is easy to show that 
for any limit point $H$ of $\{H_n\}$, $L=HK$. Therefore, $\{HK\mid H\in\Sub_G\}$ is closed. Now if $H\in\spg$, then there exists 
$\{A_n\}\subset\cP_1(G)$, such that $A_n\to H$. By Lemma \ref{subpg}\,(2), $A_nK\in\spg$. As $A_nK\to HK$ and $\spg$ is closed, 
we get that $HK\in\spg$. Now arguing as above, it is easy to show that $\{HK\mid H\in\spg\}$ is a closed subset of $\spg$. 
Moreover, if $T\in\Aut(G)$ with $T(K)=K$, then $T(HK)=T(H)K$ for any $H\in \Sub_G$. Therefore, the sets mentioned 
in (1) are invariant under the action of $\Aut(G)$ if $T(K)=K$. Thus (1) holds.

If $A\in\cP_1(G)$, then $AK$ is closed and $\pi(A)\in\cP_1(G/K)$. Conversely, suppose $B\in \cP_1(G/K$) is nontrivial. 
It is easy to see that $B_1:=\pi^{-1}(B)$ is a closed connected nilpotent Lie subgroup of $G$. There exists 
a neighbourhood $U$ of $e$ such that  every element of  $U$ is contained in a one-parameter subgroup and $U\mi K$ 
is nonempty. Choose any $x\in U\mi K$, there exists a one-parameter subgroup $\{x_t\}$ with $x_1=x$. Then 
$\pi(\{x_t\})=B$ and $B_1=\{x_t\}K$. As $K$ is central in $G$ we get that $B_1$ is, in fact, abelian. If $\{x_t\}$ is closed, 
then it belongs to $\cP_1(G)$ and $\pi(\{x_t\})=B$. Now suppose it is not closed, then its closure is compact, and hence 
$B_1$ is compact. This implies that the set of torsion elements is dense in $B_1$. Since $U\mi K$ is open, we can choose 
 $y\in B_1\cap(U\mi K)$ such that $y^n=e$, for some $n\in\N\mi\{1\}$. Let $\{y_t\}$ be a one-parameter subgroup in $B_1$ 
 such that $y_1=y$. Then $\pi(\{y_t\})=B$. Here, $\{y_t\}$ is closed as $y_1^n=e$. Therefore,  
 $\cP_1(G/K)=\{\pi(A)\mid A\in\cP_1(G)\}$, i.e.\ (2) holds. 

Let $\pi':\spg\to \spp_{G/K}$ be the map induced by $\pi$.  
Then $\pi'$ is a continuous closed map. From (2), we get that $\spp_{G/K}=
\ol{\cP_1(G/K)}=\pi'(\ol{\cP_1(G)})=\pi'(\spg)=\{\pi(H)\mid H\in\spg\}$. Thus (3) holds. 

Let $\cS(K):=\{HK\mid H\in\spg\}$. By (1), $\cS(K)$ is closed in $\spg$, and hence it is compact. 
Observe that $\cS(K)=\{H\in\spg\mid HK=KH=H\}$. By (3), $\pi'(\cS(K))=\pi'(\spg)=\spp_{G/K}$ and it is 
obvious that $\pi'|_{\cS(K)}$ is continuous and bijective, and hence it is a homeomorphism. Thus (4) holds. 
\end{proof}

\section{Some properties of automorphisms and their actions on \texorpdfstring{$G$}{G} and \texorpdfstring{${\mathrm{Sub}}^p_G$}{Sub p G}}

In this section, we discuss some properties and actions of automorphisms of a connected Lie group 
$G$ on $G$ as well as $\spg$ and discuss some preliminary results which are useful. 

For convenience, we introduce a class $\cp$ of Lie groups as follows:

\begin{defn}
A connected Lie group $G$ is said to be in class $\cp$ if $G$ has no compact central subgroup 
of positive dimension $($i.e.\ if $G$ has no non-trivial compact connected central subgroup$)$.
\end{defn} 

The notation is inspired by the fact that the class $\cp$ is a subset of the class $\mathcal{C}$ 
defined by Dani and McCrudden in \cite{DM}. The class $\mathcal{C}$ consists of connected Lie groups 
which admit a (real) linear representation with discrete kernel. Proposition 2.5 of \cite{DM} 
shows that $G$ is in class $\mathcal{C}$ if and only if for the radical $R$ of $G$, the 
intersection of $\ol{[R,R]}$ with the center of $G$ has no non-trivial compact subgroup. 
Therefore, the class $\cp$ is a subset of the class $\mathcal{C}$. Note that any connected Lie group 
$G$ admits a maximal compact connected central subgroup $M$, and $G/M$ has no 
non-trivial compact connected central subgroup and hence it belongs to class $\cp$. We may recall 
a result of Dani in \cite{D} that if $G$ belongs to class $\cp$, then $\Aut(G)$ is almost algebraic, i.e.\ 
it is a(n open) subgroup of finite index in an algebraic subgroup of $\GL(\gf)$, and in particular,  
it has finitely many connected components. 

It is known that $T\in\Aut(G)$ acts distally on $G$ if and only if $\du T$ acts distally on the Lie algebra $\gf$ of $G$, 
and the latter statement is equivalent to the statement that all the eigenvalues of $\du T$ have absolute value 1 
(cf.\ \cite{Ab2}). It is also known that if $G$ admits an automorphism $T$ such that $\du T$ has no 
eigenvalue of absolute value 1 (equivalently, if $T$ is expansive; see \cite{Sh2}, and also \cite{Bh}), then $G$ must 
be nilpotent (this is well-known, see e.g.\ \cite{Sh2}). We will need the following lemma which is slightly stronger. 
The lemma may be known, but we will give a short proof for the sake of completeness. 

\begin{lem} \label{eigen}
Let $G$ be a connected non-nilpotent Lie group and let $T$ be an automorphism of $G$. Then 
at least one eigenvalue of $\du T$ is a root of unity.
\end{lem}

 \begin{proof}
Let $M$ be the largest compact connected central subgroup of $G$. 
 Then $G/M$ is in class $\cp$. As $G$ is not nilpotent, $G/M$ is also not nilpotent. 
 Since $M$ is characteristic in $G$, we may replace $G$ by $G/M$ and assume that $G\in\cp$. 
 If $T$ has finite order, then $(\du T)^n=\Id$ for some $n\in\N$, and the assertion follows in this case. 
 Now suppose $T$ does not have finite order. 
 
 Since $G\in\cp$,  $\Aut(G)$ is almost algebraic (cf.\ \cite{D}), and 
 hence it has finitely many connected components. Now we may replace $T$ by its suitable 
 power and assume that $T\in\Aut(G)^0$, the connected component of the identity 
 in $\Aut(G)$, and that $T$ is contained in a non-trivial one-parameter subgroup (say) $\{T_t\}$
 of $\Aut(G)^0$ as $T=T_1$. We will show that 1 is an eigenvalue of $\du T$. 
 
 Suppose $G$ is solvable. Let $H=\{T_t\}\ltimes G$.  Then $H$ is solvable 
 and the nilradical of $H$, being contained in $G$, is the same as the nilradical $N$ of $G$. 
 Therefore, $H/N$ is abelian. Thus $T$ acts trivially on $G/N$. As $G$ is not nilpotent, $G\ne N$, 
 and hence 1 is an eigenvalue $\du T$. 
 
 Now suppose $G$ is not solvable. Let $R$ be the radical of $G$. Then $R$ is characteristic in $G$, and let 
 $\ol{T}_t$ be the automorphisms of $G/R$ induced by $T_t$ for each $t$. It is enough to prove that 1 is an 
 eigenvalue of $\du \ol{T}$. Hence we may replace $G$ by $G/R$ and $\{T_t\}$ by $\{\ol{T}_t\}$ 
 and assume that $G$ is semisimple.  Now $\Aut(G)^0$ consists of inner automorphisms of $G$ and
 $T_t=\inn(x_t)$, $t\in\R$, for some one-parameter subgroup $\{x_t\}\subset G$. 
 If $x_1=e$, then $T=T_1=\Id$. Now suppose $x_1\ne e$. Then $\{x_t\}$ is non-trivial.
 Let $X\in \gf$ be such that $\exp tX=x_t$, $t\in \R$. As $T(x_t)=x_t$, $t\in\R$, 
 we have that $\du T(X)=X$ and 1 is an eigenvalue of $\du T$. 
 \end{proof}

\begin{rem} \label{eigen-rem} The proof above also shows that for a connected solvable Lie group $G$ with 
the nilradical $N$, for any $T\in \Aut(G)$, $\ol{T}$ has finite order, where $\ol{T}$ is the automorphism of 
$G/N$ induced by $T$. In particular, all the eigenvalues of $\du\ol{T}$ on the Lie algebra of $G/N$ 
are roots of unity. 
\end{rem}

For a connected Lie group $G$ and $T\in\Aut(G)$, let $C(T)$ denote the contraction group of $T$ 
and it is defined as 
$$
C(T):=\{x\in G\mid T^n(x)\to e\mbox{ as } n\to\infty\}.$$ 
Note that $C(T)$ is a connected nilpotent subgroup of $G,$ and if it is closed in $G$, then it is simply 
connected (see \cite{Si}, see also \cite{HSi}). Both $C(T)$ and $C(T^{-1})$ are trivial 
if and only if $T$ acts distally on $G$ (cf.\  \cite{RS2}, Theorem 4.1). Note also that $C(T)$ is 
trivial if and only if all the eigenvalues of $\du T$ have absolute value greater than or equal to 1 
(cf.\ \cite{RS2, Ab2}). 

We will need the following lemma which can be proven easily by using Theorem 1.1 of \cite{DS1}, 
 Theorem 2.4 and Corollary 2.7 of \cite{HSi}; we will give a sketch of proof. Note that the largest compact 
 connected central subgroup of $G$ is characteristic in $G$. The lemma generalises a part of the 
 statement in Proposition 4.3 of \cite{RS2}.
  
 \begin{lem} \label{ct-closed}
 Let $G$ be a connected Lie group and let $T\in\Aut(G)$. Then the following are equivalent:
 \begin{enumerate}
 \item $T$ acts distally on the largest compact connected central subgroup of $G$.
 \item The contraction groups $C(T)$ and $C(T^{-1})$ are closed, simply connected and nilpotent subgroups of $G$. 
 \item $T$ acts distally on every compact $T$-invariant subgroup of $G$. 
 \end{enumerate}
 \end{lem}
 
 \begin{proof} Let $M$ be the largest compact connected central subgroup of $G$. If $G=M$, then it is easy to see that 
 $(1-3)$ are equivalent. Now suppose $G\ne M$. Let $\pi:G\to G/M$ be 
 the natural projection and let $T'\in \Aut(G/M)$ be the automorphism induced by $T$ on $G/M$. 
 As $G/M\in\cp$, using Theorem 1.1 of \cite{DS1}, 
 one can show that $C(T')$ is closed (see the proof of Proposition 4.3 of \cite{RS2} for more details). Hence 
 $\pi^{-1}(C(T'))=C_M(T):=\{x\in G\mid T^n(x)M\to M\mbox{ in }G/M\}$ is closed. 
  Suppose (1) holds, i.e.\ $T$ acts distally on $M$. Then $C(T)\cap M=C(T|_M)=\{e\}$. 
  As $C_M(T)$ is closed, by Corollary 2.7 of \cite{HSi}, $C(T)$ is closed. Similarly, $C(T^{-1})$ is
 closed. The rest of the assertion in (2) follows from Corollary 2.4 in \cite{Si}.
 This proves $(1)\implies (2)$. 
 
 Now suppose (2) holds, i.e.\ $C(T)$ is closed. Let $K$ be any compact $T$-invariant subgroup of $G$. 
 Then $C(T|_K)=C(T)\cap K$ is a compact subgroup of $C(T)$. Hence $C(T|_K)$ is trivial 
 (cf.\ \cite{HSi}, Corollary 2.5). Similarly, $C((T|_K)^{-1})$ is trivial and $T$ acts 
 distally on $K$ (cf.\ \cite{Ja}). Thus (3) holds. $(3)\implies (1)$ is obvious. 
 \end{proof}

We note the following useful proposition which will enable us to work on a quotient Lie group which belongs to class $\cp$.

\begin{prop} \label{gk}
Let $G$ be a connected Lie group and let $T\in\Aut(G)$. Let $K$ be a $T$-invariant compact connected 
central subgroup of $G$ and let $\ol{T}$ denote the automorphism of $G/K$ induced by $T$. 
If $T$ acts distally on $\spg$, then $\ol{T}$ acts distally on $\spp_{G/K}$. 
\end{prop}

\begin{proof} Let $T\in\Aut(G)$ and let $\pi:G\to G/K$ be the natural projection. Here,
$\ol{T}(\pi(H))=\pi(T(H))$ for any $H\in\Sub_G$, and $\spp_{G/K}$ is $\ol{T}$ invariant. 
Let $\cS(K)=\{H\in \spg\mid HK=H\}$. It is a compact $T$-invariant subset of $\spg$ and it is 
homeomorphic to $\spp_{G/K}$ under the map induced by $\pi$ on $\cS(K)$ (cf.\ Lemma \ref{subpgk}\,(4)). 
Suppose $T$ acts distally on $\spg$. Then the $T$-action on $\cS(K)$ is distal. As $\pi\circ T=\ol{T}\circ\pi$, 
we get that $\ol{T}$ acts distally on $\spp_{G/K}$. 
\end{proof}
 
 We will also need the following elementary lemma. 
 
 \begin{lem} \label{scn}
Let $G$ be a simply connected nilpotent Lie group with the Lie algebra $\gf$ and let $T\in \Aut(G)$. 
Then the set $\cP_1(G)$ consisting of all closed one-parameter subgroups of $G$ is closed in $\Sub_G$ 
and is the same as $\spg$. Moreover, $T$ acts distally on $\spg$ if and only if $\du T$ acts distally on $\spp_\gf$. 
\end{lem}

\begin{proof}
In a simply connected nilpotent Lie group, all one-parameter subgroups are closed 
and there is a one-to-one correspondence between non-trivial (closed) one-parameter subgroups of $G$
and one-dimensional subspaces of $\gf$. The correspondence is via the exponential map $\exp: \gf\to G$ with 
$\log$ as its inverse. Namely; for $\{x_t\}$, $x_t=\exp tX$, for some $X\in\gf$, and $\log x_t=tX$, $t\in\R$.
Also the exponential map induces a homeomorphism from $\cP_1(\gf)\to\cP_1(G)$. 
As $\exp\circ\,\du T=T\circ\exp$, it follows that $T$ acts distally on $\cP_1(G)=\spg$ if and
only if $\du T$ acts distally on $\cP_1(\gf)=\spp_\gf$. 
\end{proof}

\section{Proofs of the main results}

In this section, we prove the main results stated in the introduction. 
Additionally, we get some corollaries which are derived from the main results. We give some examples 
of Lie groups $G$ with a nontrivial central torus on which all automorphisms act distally, 
which is stronger than one of the conditions in Theorems \ref{distal-g}, \ref{general} and \ref{subgp-action}.

\begin{proof}[Proof of Theorem {\rm \ref{main}} $(1)$ and $(2)$ for nilpotent groups] Let $G$ 
be a (non-trivial) connected Lie group without any non-trivial compact connected central subgroup, 
i.e.\ $G$ belongs to class $\cp$. Let $T\in\Aut(G)$.

\smallskip
\noindent{$(1):$} Suppose $G$ is abelian. Then it is a vector group isomorphic to $\R^d$, 
for some $d\in\N$. Note that $\D=\{r\,\Id\mid r\in\R\mi\{0\}\}$, where $\Id$ is the identity matrix in $\GL(d,\R)$. 
If $d=1$, then $\GL(1,\R)$ is abelian, and the assertion follows trivially. Now suppose $d\ge 2$. Then $\spg$ 
is isomorphic to $\mbb {RP}^{d-1} \sqcup \{\{e\}\}$ and $T\in \GL(d,\R)$. 
Both $\mbb { RP}^{d-1}$ and $\{\{e\}\}$ are compact and $T$ invariant. 
Hence $T$ acts distally on $\spp_{\R^d}$ if and only if $T$ acts distally on $\mbb{RP}^{d-1}$. 
One can define an action of $T$ on $\mbb S^{d-1}$, the 
boundary of the unit sphere in $\R^d$, as $\ol{T}(x)=T(x)/\|T(x)\|$, where $\|\cdot\|$ denotes the 
usual norm on $\R^d$. Here, $\ol{T}$ is a homeomorphism and $\ol{T}(-x)= -\ol{T}(x)$. 
As $\RP^{d-1}$ is a quotient space of $\mbb S^{d-1}$, 
the action of $T$ on $\RP^{d-1}$ is the same as the action of $\ol{T}$ on 
${\mbb S}^{d-1}/\sim$, where $\sim$ is the equivalence relation which identifies  
$x$ with $-x$ for each $x\in \mbb S^{d-1}$. It is easy to see that $T$ acts distally 
on $\RP^{d-1}$ if and only if $\ol{T}$ acts distally on $\mbb S^{d-1}$. 
The latter statement is equivalent to the following: $T \in \K \D$, where $\K$ is a 
compact subgroup of $\GL(d,\R)$ and $\D$ is the center of $\GL(d,\R)$  
(this follows from Theorem 1 in \cite{SY1}, see also \cite{SY4}). 

Conversely, for $G=\R^d$ and $T\in \K\D$, with a compact subgroup $\K$ of $\GL(d,\R)$ and the center 
$\D$ of $\GL(d,\R)$, the action of $T$ on $\spg$ is same as that of $S$ on $\spg$ for some $S\in\K$ 
since $\D$ acts trivially on $\spg$. As the compact subgroups of $\Aut(G)$ act distally on any closed 
invariant subspace of $\Sub_G$ ($\spg$ in this case), we have that $T$ acts distally on $\spg$. 
This completes the proof of Theorem \ref{main}\,(1). 

\smallskip
\noindent{$(2):$} Suppose $G$ belongs to the class $\cp$ and it is non-abelian and nilpotent. Then
 $G$ is simply connected as $G$ has no compact central subgroup of positive dimension.
Suppose $T$ acts distally on $\spg$. We want to show 
that $T$ is contained in a compact subgroup of $\Aut(G)$. Since $G$ belongs to the class $\cp$, 
$\Aut(G)$ is almost algebraic as a subgroup of $\GL(\gf)$ (cf.\ \cite{D}). 

As $G$ is simply connected and nilpotent, by Lemma \ref{scn}, $\du T$ acts 
distally on $\spp_\gf$. Then from (1) above, we have that $\du T\in \K\{r\,\Id\mid r\ne 0\}$ 
for some compact subgroup $\K$ of $\GL(\gf)$ and $\Id$ is the identity map on $\gf$. 
Here, $\du T=\psi\,\circ\, r\,\Id$ for some $r\in \R\setminus\{0\}$ and $\psi\in \K$. If $r=\pm 1$, then 
$\du T^2\in \K$ and $T$ generates a relatively compact subgroup of $\Aut(G)$. If possible suppose 
$r\ne\pm 1$. Since $\Aut(G)$ is almost algebraic, there exists $n\in\N$ such that
$\du T^n$ is contained in a connected almost algebraic subgroup $\R\times \mbb T^k$ of 
$\Aut(G)^0$ (which is identified with its image in $\GL(\gf)$), where $\mbb T^k=(\mbb S^1)^k$ 
for some $k\in\N$. Therefore, $\psi^n\,\circ\, r^n\,\Id\in \R\times\mbb T^k\subset \Aut(G)^0$. 
Hence both $\psi^n$ and $r^n\,\Id\in \Aut(G)^0$. But $r^n\,\Id$ is not a Lie algebra automorphism of $\gf$ 
as $r^n\ne 1$; otherwise $r^n[X,Y]=r^n\Id[X,Y]=[r^n\Id(X),r^n\Id(Y)]=r^{2n}[X,Y]$, and it would imply that 
$[X,Y]=0$ for all $X,Y\in\gf$, which leads to a contradiction as $G$ is not abelian. Thus, $r=\pm 1$ and 
$T$ is contained in a compact subgroup of $\Aut(G)$. The converse statement is obvious. 
Thus (2) holds if $G$ is nilpotent.
\end{proof} 

We will continue with the proof of Theorem \ref{main} after proving Theorems \ref{distal-g} and \ref{unip}.  

\begin{proof}[Proof of Theorem {\rm\ref{distal-g}}]  Let $G$ be a connected Lie group which is not a vector group. 
Let $M$ be the largest compact connected central subgroup of $G$. Suppose 
$T\in\Aut(G)$ be such that it acts distally on $M$. Suppose $T$ acts distally on $\spg$. 
Then we want to show that $T$ acts distally on $G$. The assertion holds trivially if $G=M$. 

\smallskip
\noindent{\bf Step 1:} Suppose $G$ is simply connected and nilpotent. Since $G$ is not a vector group, 
$G$ is non-abelian and by the proof of Theorem \ref{main}\,(2) for the nilpotent case, $T$ is contained in 
a compact subgroup of $\Aut(G)$. Hence $T$ acts distally on $G$. 

\smallskip
\noindent{\bf Step 2:} Now suppose $G$ is not a simply connected nilpotent group.
 By Theorem 4.1 of \cite{RS2}, 
it is enough to prove that the contraction groups $C(T)$ and $C(T^{-1})$ are trivial. As $T$ acts distally on the 
maximal compact connected central subgroup $M$, by Lemma \ref{ct-closed}, both 
$C(T)$ and $C(T^{-1})$ are closed, simply connected and nilpotent subgroups. Suppose $C(T)$ is non-trivial. 
If possible, suppose $C(T)$ is non-abelian. As it is $T$-invariant and $T|_{C(T)}$ acts distally on $\spp_{C(T)}$, 
we get from Step 1 that $T|_{C(T)}$ acts distally on $C(T)$, a contradiction. Thus $C(T)$ is abelian and 
 $C(T)$ is isomorphic to $\R^d$ as it is simply connected. By Theorem \ref{main}\,(1),
 $T_1=T|_{C(T)}=\phi\circ \ap\,I$ for 
 some $\phi$ contained in a compact subgroup of $\GL(d,\R)$ and $\ap\in \R$, $0< |\ap|<1$; here $I$ denotes the 
 identity matrix in $\GL(d,\R)$, replacing $T$ by $T^2$, we may assume that $\ap>0$.

\smallskip
\noindent{\bf Step 3:} Suppose $G$ is nilpotent but not simply connected. Then $M$ as above is non-trivial. Since $T$ acts 
distally on $M$, by Lemma 2.5 of \cite{Ab2}, we have that $T^k|_M$ is unipotent for some $k\in\N$, i.e.\ $\du T$ has an 
eigenvalue which is a root of unity. If $G$ is not nilpotent, by Lemma \ref{eigen}, $\du T$  has an eigenvalue which is a 
root of unity. We may replace $T$ by $T^m$, for some $m\in\N$, and assume that 1 is an eigenvalue of $\du T$. 
Let $X\in \gf$, $X\ne 0$ be such that $\du T(X)=X$ and let $A=\{\exp tX\}$, a one-parameter 
subgroup, it is either closed or its closure $\bar A$ is compact. Since  $T|_A=\Id$, we have that 
$T|_{\bar A}=\Id$ and it follows that $\bar A$ normalises $C(T)$ and $\bar A\cap C(T)=\{e\}$. 

We now show that $C(T)$ normalises $\bar A$, which would imply that $\bar A\, C(T)$ is 
abelian as $C(T)$ and $A$ are so. Here, $T(\bar A)=\bar A$ and $\bar A\in \spg$. Now for any nonzero $c\in C(T)$, 
we have that 
 $$
 T^n(c\bar A c^{-1})=T^n(c)\,\bar A\,T^n(c^{-1})\to\bar A.$$
 Since $T$ acts distally on $\spg$, we get that $c\bar Ac^{-1}=\bar A$, for all $c\in C(T)$. 
 Thus $C(T)$ normalises $\bar A$, and hence $\bar A\, C(T)$ is abelian. 
 
 Let $H=\ol{A\times C(T)}$. Then $H$ is a closed $T$-invariant abelian subgroup and $\bar A\subset H$. First, suppose $H$ 
 is a vector group. By Theorem \ref{main}\,(1),  $T|_H\in \K\D$, for some compact subgroup $\K$ in $\GL(H)$ and 
 the center $\D$ of $\GL(H)$, and thus all the eigenvalues of $T|_H$ have the same absolute value. Since $T|_{\bar A}=\Id$ 
 and $T|_{C(T)}$ have eigenvalues of absolute value $\ap\ne 1$, it leads to a contradiction. Thus $C(T)$ must be trivial in
 this case. 
 
 \smallskip
 \noindent{\bf Step 4:}  Now suppose $H$ is not a vector group. Since $C(T)$ is closed and the image of $A$ 
 is dense in $H/C(T)$, we get that $H=K\times C(T)$ for the maximal compact subgroup $K$ of $H$. 
 As $T$ acts trivially on $H/C(T)$ and $K$ is $T$-invariant, we get that $T|_K=\Id$. 
 
 Let $\psi, S: H\to H$ be such that $S|_K=\Id$, $S|_{C(T)}=\ap\,\Id$, $\psi|_K=\Id$ and 
$\psi|_{C(T)}=\phi$, where $\alpha$ and $\phi$ are as in Step 2. Now $T|_{H}=S\psi=\psi S$, and $\psi$ is 
contained in a compact subgroup of $\Aut(H)$. Therefore, $S$ acts distally on $\spp_{H}$ (cf.\ \cite{SY2}, Lemma 2.2), 
and $C(S)=C(T)$. 

Let $\{z_t=x_tc_t\}$ be a closed one-parameter subgroup, where  $\{x_t\}\subset K$ is closed and isomorphic 
to $\Sc^1$ and $\{c_t\}\subset C(S)=C(T)$ is a non-trivial (closed) one-parameter subgroup such that 
$S(c_t)=c_{\ap t}$, $t\in\R$ (as addition is the group operation on $C(T)=\R^d$, we choose $c_t=tc$, $t\in\R$,  
for some nonzero  $c\in \R^d$). Let $\{n_k\}$ be an unbounded sequence in $\N$ such that 
$S^{n_k}(\{z_t\})\to C$ for some $C$ in $\spp_{H}$. Note that 
$\{z_t\}\subset \{x_t\}\times \{c_t\}$ and the latter is closed and $S$-invariant. 
So $C\subset \{x_t\}\times\{c_t\}$. For a fixed $t\in\R$, as $c_t\in C(S)=C(T)$, we get that
$$
S^{n_k}(z_t)=x_t S^{n_k}(c_t)\to x_t,$$ 
 as $n_k\to\infty$. Therefore, $\{x_t\}\subset C$. Again for a fixed $t\in\R$, 
$$S^{n_k}(z_{\ap^{-n_k}t})=
x_{\ap^{-n_k}t}S^{n_k}(c_{\ap^{-n_k}t})=x_{\ap^{-n_k}t}c_t\to x_sc_t$$ 
along a subsequence of $\{n_k\}$, for some $s\in\R$. 
Therefore, $x_sc_t$, and hence, $c_t$ belongs to $C$. Thus $\{c_t\}\subset C$ and  
$C=\{x_t\}\times \{c_t\}$, and hence $C$ is $S$-invariant. 
This implies that $S^{n_k}(S(\{z_t\}))\to C$, and hence that $S(\{z_t\})=\{z_t\}$ since $S$ acts distally on $\spg$. 
As $\{z_t\}$ is isomorphic to $\R$, there exists $\beta\in\R\mi\{0\}$ such that for all $t\in\R$, 
$$
S(z_t)=z_{\beta t}=x_{\beta t}c_{\beta t}\  \mbox{ and }\  
S(z_t)=S(x_tc_t)=x_tc_{\ap t}.$$ 
As $\{x_t\}\times \{c_t\}$ is a direct product and $\{c_t\}$ is isomorphic to $\R$, it follows that 
$\beta=\alpha$ and $x_t=x_{\ap t}$, i.e.\ $x_{(1-\ap)t}=e$ for all $t\in\R$. As $\ap\ne 1$, 
$x_t=e$ for each $t$, a contradiction. Therefore, $C(T)$ is trivial. 

Similarly, we can show that $C(T^{-1})$ is trivial, and hence $T$ acts distally on $G$. 
\end{proof}

In view of Theorem 1.1 of \cite{Ab2}, the following corollary is an easy consequence of Theorem \ref{distal-g}. 
We omit the proof.

\begin{cor} \label{distal-subgp} Let $G$ be a connected Lie group which is not a vector group and let $H$ be a 
subgroup of $\Aut(G)$ such that it acts distally on the largest compact connected central subgroup of $G$. 
If $H$ acts distally on $\spg$, then $H$ acts distally on $G$. 
\end{cor}

Before proving Theorem \ref{unip}, we first prove the following proposition which is a special case of the theorem. 

\begin{prop} \label {crux}
Let $G$ be a connected Lie group, $T\in\Aut(G)$ and let $L$ be a closed connected nilpotent normal $T$-invariant 
subgroup of $G$. Suppose $T$ acts trivially on both $L$ and $G/L$. If $T$ acts distally on $\spg$, then $T=\Id$. 
\end{prop}

\begin{proof} Since $T$ acts trivially on $G/L$ and $L$, we have that $T$ acts distally on $G$. In fact, all the eigenvalues 
of $\du T$ are equal to 1 and $T$ is unipotent. Suppose $T$ acts distally on $\spg$. Let $U$ be a neighbourhood 
of the identity $e$ such that every element of $U$ is contained in a one-parameter subgroup. Since $G$ is 
connected, $U$ generates $G$ and it is enough to show that $T(x)=x$ for all $x\in U$. 
 
\smallskip
 \noindent{\bf Step 1:} If possible, suppose $T(x)\ne x$ for some $x\in U$. There exists a one-parameter 
 subgroup $\{x_t\}$ in $G$ such that $x_1=x$. Let $H=\ol{\{x_t\}L}$. Then $H$ is a closed connected solvable $T$-invariant 
 subgroup of $G$ and $H/L$ is abelian. Let $\pi:G\to G/L$ be the natural projection.  
 Then $\pi(H)=\ol{\pi(\{x_t\})}$. If $\pi(\{x_t\})$ is unbounded, it is closed and it is isomorphic to $\R$ and in this case 
 $\{x_t\}$ is also unbounded and closed. Moreover $H=\{x_t\}L$.
 
Now suppose $\pi(\{x_t\})$ is relatively compact. As $H$ is solvable, there exists a maximal compact 
(connected, abelian) subgroup $K$ in $H$ such that $H=KL$. As $T(x)\ne x$ and $T|_L=\Id$, 
we have that $T(h)\ne h$ for some $h\in K$, and hence $T(h)\ne h$ for some finite order element $h$ in $K\cap U$.
There exists a closed (compact) one-parameter subgroup $\{h_t\}$ in $K$ such that $h_1=h$. Replacing 
$\{x_t\}$ by $\{h_t\}$, we have in this case too that $H$ is closed and $H=\{x_t\}L$.

Now we have that $H=\{x_t\}L$ is closed, where either $\pi(\{x_t\})$ (as well as $\{x_t\}$) is isomorphic to $\R$ 
or $\{x_t\}$ is isomorphic to $\Sc^1$ and $T(x_1)\ne x_1$. As $T$ acts trivially on $G/L$, we have that $T(H)=H$ 
and $T(x_t)=x_ty_t$, for some $y_t\in L$, $t\in\R$, where $y_1\ne e$. 

 \smallskip
 \noindent{\bf Step 2:} Let $A:=\{x_t\}$ and let $\{n_k\}\subset\N$ be an unbounded sequence 
such that $T^{n_k}(A)\to B$ for some $B\in\spp_H$. By Lemma \ref{discrete}, $B$ is not discrete as $A$ 
is connected and non-trivial. Let $b\in B$ be such that $b\ne e$. There exist $t_k\in\R$, $k\in\N$, 
such that 
$$
T^{n_k}(x_{t_k})=x_{t_k}y_{t_k}^{n_k}\to b.$$ 
 As $y_t\in L$, $t\in\R$, we get that $\pi(x_{t_k})\to\pi(b)$. If $H/L$ is isomorphic to $\R$, 
then we get that the sequence $\{t_k\}$ is bounded. In the second case where $\{x_t\}$ is compact, we 
can choose $\{t_k\}$ to be bounded. Passing to a subsequence if necessary, we may assume that 
$t_k\to t_0$. Therefore, $x_{t_k}\to x_{t_0}$ and $y_{t_k}\to y_{t_0}$ and 
$b=x_{t_0}y'$, where $y_{t_k}^{n_k}\to y'$ in $L$. 

 \smallskip
 \noindent{\bf Step 3:} We first assume that $L$ is simply connected. We show that $y_{t_0}=e$ in this case.
Let $\lf$ be the Lie algebra of $L$. Since $L$ is simply connected and nilpotent, $\exp: \lf\to L$ is a 
homeomorphism with $\log$ as its inverse. Let $y_{t_k}=\exp Y_k$ for some $Y_k\in\lf$. 
Then $y_{t_k}^{n_k}=\exp(n_kY_k)$ and as $y_{t_k}^{n_k}\to y'$, we have 
that $n_kY_k\to \log y'$ in $\lf$. As $\lf$ is isomorphic to $\R^d$ for some $d\in\N$, it follows that $Y_k\to 0$ in $\lf$, 
and hence $y_{t_k}\to e$, i.e.\ $y_{t_0}=e$. Now
$$
T(b)=T(x_{t_0})y'=x_{t_0}y_{t_0}y'=x_{t_0}y'=b.$$
Thus $T(B)=B$. As $T$ acts distally on $\spg$, we have that $T(A)=A$. Since $A$ is a one-parameter subgroup in $G$ 
and $T$ is unipotent, we get that $T(x_t)=x_t$, $t\in\R$. This leads to a contradiction as $T(x_1)\ne x_1$. 
Therefore, $T(x)=x$ for all $x\in U$, and hence, for all $x\in G$. Thus $T=\Id$ if $L$ is simply connected. 
 
 \smallskip
 \noindent{\bf Step 4:} Suppose that $L$ is not simply connected. Let $M$ be the maximal 
 compact (normal) subgroup of $L$. Then $M$ is central in $G$ and $L/M$ is simply connected. 
 Let $\varrho:G\to G/M$ be the natural projection and let $\ol{T}$ be the automorphism on $G/M$ 
 induced by $T$. Then $\ol{T}$ acts trivially on both $\varrho(L)$ and $\varrho(G)/\varrho(L)$. 
 By Proposition \ref{gk}, $\ol{T}$ acts distally on $\spp_{G/M}$. From Step 3, we get that $\ol{T}$ acts trivially on 
 $G/M$, i.e.\ $T(g)\in gM$, $g\in G$. In particular, for $\{x_t\}$ as above, $T(x_t)=x_ty_t$, where $y_t\in M$, 
 for all $t\in\R$, and $y_1\ne e$. Since $M$ is central in $G$, $y_tx_s=x_sy_t$, $t, s\in \R$, and $\{y_t\}$ is a 
 one-parameter subgroup of the compact central subgroup $M$. Note that $\{y_t\}$ is closed if $\{x_t\}$ is compact. 

As $\spg$ is compact, there exists an unbounded sequence $\{n_k\}$ in $\N$ such that $T^{n_k}(\{x_t\})\to C$ 
for some $C\in\spg$. We show that $C=\{x_t\}\times \ol{\{y_t\}}$. As $\{x_t\}\subset \{x_t\}\times \ol{\{y_t\}}$, 
and the latter is $T$-invariant, we get that $C\subset \{x_t\}\times \ol{\{y_t\}}$. For a fixed $t\in\R$, 
$$
T^{n_k}(x_{t/n_k})=x_{t/n_k}y_{t/n_k}^{n_k}=x_{t/n_k}y_t\to y_t$$ 
as $n_k\to\infty$. Therefore, $y_t\in C$ for each $t$, and $\ol{\{y_t\}}\subset C$ as $C$ is closed. Also, 
$T^{n_k}(x_t)=x_ty_t^{n_k}\to x_ty$ along a subsequence of $\{n_k\}$ for some $y\in\ol{\{y_t\}}\subset C$.
Hence, $x_t\in C$, $t\in\R$. Thus $\{x_t\}\times \ol{\{y_t\}}\subset C$. 
Now $T(C)=C$, and we get that $T^{n_k}(T(\{x_t\}))\to C$. As $T$ acts distally on $\spg$, we have that 
$T(\{x_t\})=\{x_t\}$, and since $T$ is unipotent, $T(x_t)=x_t$ for each $t$. 
But $T(x_1)=x_1y_1\ne x_1$.  
This leads to a contradiction. Therefore, $T(x)=x$ for each $x\in U$, and hence, for each $x\in G$. 
Thus $T=\Id$.
 \end{proof}

\begin{proof}[Proof of Theorem \rm{\ref{unip}}]
Let $G$ be a connected Lie group and let $T\in\Aut(G)$. Suppose $T$ is unipotent and it acts distally on $\spg$. 
We show that $T=\Id$. 

\smallskip
\noindent{\bf Step 1:} Suppose $G$ is nilpotent. If possible, suppose $T\ne\Id$. Since $T$ is unipotent, 
by Proposition 3.10 of \cite{SY3}, there exists an increasing sequence of closed connected normal 
T-invariant subgroups $\{e\} = G_0 \subset G_1 \subset \cdots\subset G_n = G$, $n \geq 2$, such that 
$T$ acts trivially on each successive quotient group $G_i/G_{i-1}$, $1\leq i\leq n$, and $T$ does not act 
trivially on $G_{i+1}/G_{i-1}$, $1\leq  i \leq n -1$. Then $G_2\ne G_1$. As $T$ acts distally on $\spg$, 
it acts distally on $\spp_{G_2}$. Now by Proposition \ref{crux}, $T|_{G_2}=\Id$, 
 which leads to a contradiction. Therefore, $T=\Id$ when $G$ is nilpotent. 

\smallskip
\noindent{\bf Step 2:} Suppose $G$ is not nilpotent. Let $N$ be the nilradical of $G$. As $T$ acts distally on $\spp_N$, 
it follows from Step 1 that $T|_N=\Id$. Suppose $G$ is solvable.  As noted in Remark \ref{eigen-rem}, we have that the 
automorphism induced by $T$ on $G/N$ is a finite order automorphism, and it is trivial since it is unipotent. 
Now by Proposition \ref{crux}, $T=\Id$. 

\smallskip
\noindent{\bf Step 3:} Suppose $G$ is not solvable. Let $R$ be the radical of $G$. Then as $T$ acts 
distally on $\spp_R$, we have from Step 2 that $T|_R=\Id$. 

Now $G/N$ is reductive. Let $\pi:G\to G/N$ be the natural projection. Then  $\pi(G)=SR'$, 
an almost direct product of a 
semisimple Levi subgroup $S$ of $\pi(G)$ and the radical $R'=\pi(R)$ of $\pi(G)$, where $R'$ is central in $\pi(G)$. 
Let $\T$ be the automorphism of $\pi(G)$ induced by $T$. Then $\T$ acts trivially on $R'$. Also
$\T$ keeps $S$ invariant as it is the commutator subgroup of $\pi(G)$. Now if $\T$ acts trivially on
$\pi(G)=G/N$, then by Proposition \ref{crux}, $T=\Id$. If possible suppose $\T$ does not act trivially on $\pi(G)$.
As $\T$ is unipotent, there exists a nontrivial unipotent element $u$ in $S$ such that $\T=\inn(u)$. Considering the Iwasawa 
decomposition $S=KAU$, where $u\in U$, we have that $u$ is not in the center of $AU$ (see Lemma 6.50 of \cite{Kn}). 
Now $AUR'$ is connected, solvable and $\T$ invariant. 
Let $B$ be the closure of $\pi^{-1}(AUR')$. It is a closed connected solvable $T$-invariant subgroup of $G$. 
As $T$ acts distally on $\spp_B$, by Step 2, $T|_B=\Id$. This implies that $\T$ 
acts trivially on $AU$ and that $u$ centralises $AU$, a contradiction. Thus $\T$ acts trivially on $\pi(G)=G/N$, 
and as noted above, by Proposition \ref{crux}, we get that $T=\Id$. 

The converse statement in the theorem is obvious.
\end{proof}

\begin{proof}[Proof of Theorem \ref{main}\,$(2)$ for non-nilpotent case]
Let $G$ be a connected Lie group in class $\cp$. Then $\Aut(G)$ is almost algebraic as a subgroup 
of $\GL(\gf)$ (via the correspondence $T\leftrightarrow \du T$) (cf.\ \cite{D}). Suppose $T\in\Aut(G)$ acts 
distally on $\spg$. Suppose $G$ is not nilpotent. In particular, it is not a vector group.  
By Theorem \ref{distal-g}, $T$ acts distally on $G$. Then for some $n\in\N$, $T^n$ is contained in a connected 
abelian almost algebraic subgroup of $\Aut(G)$ and it has the form $\K\times\,\U$, where $\K,\,\U\subset\Aut(G)$, 
$\K$ is an abelian compact connected group and $\U$ is an abelian unipotent connected group; i.e.\ it consists of 
unipotent elements (cf.\ \cite{Ab1}, Corollaries 2.3 and 2.5). Now $T^n=S\circ\psi=\psi\circ S$, where $S\in\U$ is unipotent 
and $\psi\in\K$. Hence, $S$ acts distally on $\spg$ (cf.\ \cite{SY2}, Lemma 2.2). By Theorem \ref{unip}, $S=\Id$. 
Thus $T^n=\psi\in\K$, and hence $T$ is contained in a compact subgroup of $\Aut(G)$. 
The converse is obvious (cf.\ \cite{PaS}, Lemma 4.2).
\end{proof}

Now we want to prove Theorem \ref{general}.  We need an elementary result about certain automorphisms keeping 
a maximal compact connected abelian subgroup (maximal torus) invariant.  In a connected Lie group $G$, all maximal tori 
are conjugate to each other. Any compact abelian subgroup of $G$ is contained in a maximal torus, as    
any maximal compact abelian subgroup of $G$ is connected. If $A$ is a maximal torus in $G$, then so is $T(A)$, 
for any $T\in\Aut(G)$, and $T(A)$ is conjugate to $A$. Any maximal compact subgroup of a connected solvable Lie 
group is abelian (cf.\ \cite{I}).
 
 In \cite{DS2}, any subgroup of $\Aut(G)$ which fixes elements of a maximal torus is shown to be almost algebraic. 
 Given any automorphism $T$ of $G$, it need not keep any torus invariant. However, this is the case if $T$ is contained 
 in a compact connected abelian subgroup of $\Aut(G)$, as shown in the following elementary lemma. The lemma also 
 holds if we replace the radical $R$ by any closed connected characteristic subgroup of $G$. 

\begin{lem} \label{aut-inv} Let $G$ be a connected Lie group. If $\K$ is a compact connected abelian subgroup 
of $\Aut(G)$, then there exists a maximal torus of $G$ \rm{(}resp.\ of the radical of $G$\rm{)}  on which every 
automorphism in $\K$ acts trivially. 
\end{lem}

\begin{proof} 
Let $A$ be any maximal torus in a connected Lie group $G$. It is easy to see that for any closed connected normal 
subgroup $L$ of $G$, $A\cap L$ is a maximal torus of $L$. 

Let $\K$ be a compact connected abelian subgroup of $\Aut(G)$ and let $H=\K\ltimes G$; where the action of $\K$ 
on $G$ is by automorphisms. Then $H$ is a connected Lie group and $G$ is normal in $H$. Let $R$ be the radical 
of $G$. Then $R$ is characteristic in $G$ and hence normal in $H$. Let $A$ be a maximal torus in $H$ containing 
$\K$. Arguing as above, we get that $A\cap G$ (resp.\  $A\cap R$) is a maximal torus in $G$ (resp.\ $R$). As 
$\K\subset A$ and $A$ is abelian, we get that every automorphism in $\K$ acts trivially on $A\cap G$ (resp.\ $A\cap R$). 
\end{proof}

\begin{proof}[Proof of Theorem \ref{general}] Let $G$ be a connected Lie group and let $T\in\Aut(G)$. 
The statements $(5)\implies (4)\implies (3)\implies (2)\implies (1)$ are obvious. 
Now it is enough to show that $(1)\implies (5)$ when $G$ is not a vector group and 
$T$ acts distally on the largest compact connected central subgroup (say) $M$ of $G$.
Suppose (1) holds, i.e.\ $T$ acts distally on $\spg$. Then we want to show that $T$ is contained 
in a compact subgroup of $\Aut(G)$. If $M$ is trivial, then $G$ belongs to the class $\cp$ and the 
assertion follows from Theorem \ref{main}\,(2). Now suppose $M$ is non-trivial. By Theorem \ref{distal-g}, 
$T$ acts distally on $G$. Then by Theorem 3.1 of \cite{RS1}, $\ol{T}$ acts distally on $G/M$, where 
$\ol{T}$ is the automorphism of $G/M$ induced by $T$. 

Suppose $G=M$. Since $T$ acts distally on $M$, by Lemma 2.5 of \cite{Ab2}, 
$T^n$ is unipotent for some $n\in\N$. As $T$ acts distally on $\spg$, so does $T^n$, 
and by Theorem \ref{unip}, $T^n=\Id$. Hence the assertion holds when $G$ is compact and abelian. 

Now suppose $M$ is non-trivial and $G\ne M$. We show that there exists a maximal torus of $G$ 
which is $T^k$-invariant for some $k\in\N$. Let $\pi:G\to G/M$ be the natural projection. 
By Proposition \ref{gk}, $\ol{T}$ acts distally on $\spp_{G/M}$. 

Suppose $G/M$ is a vector group. Then by Theorem \ref{main}\,(1), 
$\ol{T}=\psi\circ\,\alpha\,\Id$, for some $\ap\in\R\setminus\{0\}$ and $\psi$ 
generates a relatively compact subgroup in $\Aut(G/M)$. Since $\ol{T}$ acts distally on $G/M$, 
we get that $|\alpha|=1$, and hence $\ol{T}$ is contained in a compact subgroup of $\Aut(G/M)$. 
Now suppose $G/M$ is not a vector group. As $G/M$ belongs to the class $\cp$, by Theorem \ref{main}\,(2), 
$\ol{T}$ is contained in a compact subgroup of $\Aut(G/M)$. So in either case we have that $\ol{T}$ is contained 
in a compact subgroup of $\Aut(G/M)$.  

Let $\K$ be the connected component of the identity in the closed (compact) subgroup generated by 
$\ol{T}$ in $\Aut(G/M)$. Then $(\ol{T})^k\in \K$ for some $k\in\N$. 
Now by Lemma \ref{aut-inv}, $(\ol{T})^k$ keeps a maximal torus (say) $A$ of $G/M$ invariant and 
it acts trivially on $A$. Let $K=\pi^{-1}(A)$. Then $K$ is a maximal torus of $G$ and $T^k(K)=K$. 

Note that $K$ is compact, connected and abelian. Moreover, as $T$ acts distally on $G$, so does $T^k$ 
and the restriction of $T^k$ to $K$ is also distal. Arguing as above for $K$ instead of $M$, we get that for 
some $m\in\N$, $T^m|_K$ is unipotent and, as it acts distally on $\spp_K$, it acts trivially on $K$ 
(by Theorem \ref{unip}). Without loss of any generality, we may replace $T$ by $T^m$ and assume that 
$T$ fixes every element of $K$, 
i.e.\ $T\in F_K(G):=\{\varrho \in \Aut(G)\mid \varrho(x)=x \mbox{ for all } x\in K\}$. We know that $F_K(G)$ is almost 
algebraic (cf.\ \cite{DS2}). As $T$ acts distally on $G$ and $T\in F_K(G)$, for some $n\in\N$, 
$T^n$ is contained in a connected almost algebraic subgroup $\Bc\times \U$ of $F_K(G)$, where $\Bc$ 
is a compact group and $\U$ consists of unipotent elements (cf.\ \cite{Ab1}, Corollaries 2.3 and 2.5). 
Now we have that $T=\phi\circ S=S\circ\phi$ such that $\phi\in \Bc$ 
and $S\in \U$. As $\Bc$ is compact, we have that $S$ acts distally on $\spg$. 
By Theorem \ref{unip}, $S=\Id$. Then $T=\phi\in\Bc$ and (5) holds. 
\end{proof}

Using Theorems \ref{main} and \ref{general} and some known results, we prove Theorem \ref{subgp-action}.

\begin{proof}[Proof of Theorem \ref{subgp-action}] Let $G$ be a connected Lie group and let $\Hc$ 
be a subgroup of $\Aut(G)$. Suppose $\Hc$ acts distally on the largest compact connected central 
subgroup of $G$. Then so does $\ol{\Hc}$. Note that 
$\Hc$ acts distally on a compact $\Hc$-invariant subspace of $\Sub_G$ if and only if so does $\ol{\Hc}$ 
(both the above statements follow from Theorem 1 in \cite{E}). Also (7) holds for $\Hc$ if and only if it holds  
when $\Hc$ is replaced by $\ol{\Hc}$, as $\K\D$ is closed in $\GL(d,\R)$ for $\K,\D$ as in (7).  
Therefore, to prove the assertions, we may assume that $\Hc$ is closed. 

Note that every compact subgroup of $\Aut(G)$ acts distally on $\Sub_G$ (see e.g.\ \cite{SY2}, Lemma 2.2). 
Now $(6)\implies (5)\implies (4)\implies (3)\implies (2)\implies (1)$ are obvious. 

First suppose that $G$ is not a vector group. We show that $(1)\implies (6)$. 
Suppose (1) holds, i.e.\ every $T\in \Hc$ acts distally on $\spg$. By Theorem \ref{general}, 
every $T\in \Hc$ is contained in a compact subgroup of $\Aut(G)$. Recall that $\Aut(G)$ is identified 
with a closed subgroup of $\GL(\gf)$ via the map $T\mapsto\du T$, where $\gf$ is the Lie algebra of $G$ 
and it is a $d$-dimensional vector space for some $d\in\N$  (cf.\ \cite{Ch}, see also \cite{Hoc}). 
Therefore, $\Hc$ is also a closed subgroup of $\GL(\gf)$ and the latter is isomorphic to $\GL(d,\R)$, 
and we get by Theorem 1.1 of \cite{FP} (see also Proposition 2 of \cite{SY1}) that $\Hc$ is compact. 
Thus (6) holds and $(1-6)$ are equivalent. 
Thus, if $G$ is not a vector group, the assertions in the theorem hold for this case.  
 
Now suppose $G=\R^d$ for some $d\in\N$. Suppose (3) holds, i.e.\ $\Hc$ acts distally on $\ol{\Sub^c_G}$. 
Then $T\in\rm{(NC)}$ for every $T\in\Hc$. Since $G$ belongs to the class $\cp$, by Theorem 4.1 of \cite{SY3}, 
$T$ is contained in a compact subgroup of $\Aut(G)$ for every $T\in\Hc$. Since $\Hc$ is closed, arguing as 
above using Theorem 1.1 of \cite{FP} or Proposition 2 of \cite{SY1}, we get that $\Hc$ is compact. Hence
(6) holds. Thus $(3-6)$ are equivalent when $G$ is a vector group. 
Now we show that (1), (2) and (7) are equivalent. Suppose (7) holds. Note that 
$\D=\{r\,\Id\in\GL(d,\R)\mid r\in\R\mi\{0\}\}$. 
Therefore, $\D$ acts trivially on $\spg$ and the action of $\K\D$ on $\spg$ is same as that 
of the compact group $\K$ on $\spg$. Therefore, the action of $\K\D$, and hence, the action of 
$\Hc$ on $\spg$ is distal. Thus $(7)\implies (2)$. We have already noted that $(2)\implies (1)$. 

Now we show that $(1)\implies (7)$. Suppose every $T\in \Hc$ acts distally on $\spg$. Let $T\in \Hc$ be fixed. 
By Theorem \ref{main}\,(1), $T\in \Bc\D$, for some compact group $\Bc\subset\GL(d,\R)$. Let 
$\Hc_1:=\ol{[\Hc,\Hc]}$. Then $\Hc_1\subset \Hc$ as the latter is closed. Now let $T\in\Hc_1$. Then $\det T=1$.
From above we have that $T=\phi\,\circ\,r\,\Id$ for some $\phi\in\Bc$ and some 
$r\in\R\mi\{0\}$. Since $\Bc$ is a compact group, $\det\phi=\pm 1$, and hence $r=\pm 1$. Therefore, 
$T$ generates a relatively compact group, for every $T\in \Hc_1$. As $\Hc_1$ is a closed subgroup of 
$\GL(d,\R)$ by Theorem 1.1 of \cite{FP} or Proposition 2 of \cite{SY1}, $\Hc_1$ is compact. 
Let $N(\Hc_1)$ be the normaliser of $\Hc_1$ in $\GL(d,\R)$. Then $N(\Hc_1)$ is a closed algebraic 
subgroup of $\GL(d,\R)$ and $\Hc,\D\subset N(\Hc_1)$. Moreover, $N(\Hc_1)\subset N(\Hc_1\D)$. 
Being algebraic, $N(\Hc_1)$ has finitely many connected components, and hence so does 
$N(\Hc_1)/\Hc_1\D$. In particular any closed abelian subgroup of $N(\Hc_1)/\Hc_1\D$ is compactly generated 
(cf.\ \cite{HN}). Let $\rho: N(\Hc_1)\to N(\Hc_1)/\Hc_1\D$ be the natural projection. Then $\rho(\Hc)$ is abelian, 
and hence so is $\ol{\rho(\Hc)}$. Therefore, $\ol{\rho(\Hc)}$ is compactly generated, and it has a unique maximal 
compact subgroup, say $\mathcal{L}$. 

Now let $T\in \Hc$. Then $T=\psi\,\circ\,r\,\Id$, for some $r\in\R\setminus\{0\}$ and some $\psi$ which generates 
a relatively compact group and $r\,\Id\in \D$. Then $\rho(T)=\rho(\psi)$ also generates a relatively compact 
group in $\ol{\rho(\Hc)}$, and hence it is contained in $\mathcal{L}$. Thus $\rho(\Hc)\subset \mathcal{L}$ and $\ol{\rho(\Hc)}$ 
is compact. As $\Hc_1$ is compact, it follows that $\ol{\Hc\D}/\D$ is compact. Since $\D$ 
has two connected components, we get that $\ol{\Hc\D}/\D^0$ is also compact and it is a Lie group. 
As $\D^0$ is a vector group which is central in $\ol{\Hc\D}$, by Lemma 3.7 of \cite{I}, 
$\ol{\Hc\D}=\K\times \D^0$, where $\K$ is the maximal compact subgroup of $\ol{\Hc\D}$. 
In particular, $\Hc\subset \K\D^0$. Hence, (7) holds. Thus, if $G$ is a vector group, (1), (2) and (7) 
are equivalent. This completes the proof. 
\end{proof}

Now we state two corollaries which are easy to deduce using results proven above and a result in \cite{Ab1}.

\begin{cor} \label{distal-gspg} Let $G$ be a connected Lie group and let $\Hc$ be a subgroup of $\Aut(G)$. 
Then $\Hc$ acts distally on both $G$ and $\spg$ if and only if $\ol{\Hc}$ is compact. 
\end{cor}

Corollary \ref{distal-gspg} is a direct consequence of Theorem \ref{subgp-action} in case $G$ is not a vector group. 
If $G$ is a vector group, then as $\Hc$ acts distally on $G$, the eigenvalues of every element in $\Hc$ 
have absolute value 1 (cf.\ \cite{Ab1}), and this, together with the statement (7) in Theorem \ref{subgp-action},  
is equivalent to the statement that $\ol{\Hc}$ is compact. We omit the details. 

The following corollary is an extension of Corollary 4.6 of \cite{SY3} for a particular class of Lie groups.
 It can be proven easily using Theorems \ref{main} and \ref{subgp-action} and Proposition \ref{gk}. 
 We omit the proof. 

\begin{cor} \label{bounded-orbits}
Let $G$ be a connected Lie group such that $G/M$ is not a vector group, where $M$ is the maximal compact 
connected central subgroup of $G$. Let $T\in\Aut(G)$ {\rm (}resp.\ $\Hc$ be a subgroup of $\Aut(G)${\rm )}. If $T$ 
(resp.\ $\Hc$) acts distally on $\spg$, then $T$ (resp.\ $\Hc$) has bounded orbits; 
i.e.\ $\{T^n(x)\mid n\in\Z\}$ (resp.\ $\Hc x$) is relatively compact in $G$, for every $x\in G$.
\end{cor}

Now we prove Corollary \ref{inner}. The statement in the corollary which requires a proof is $(1)\implies (5)$ 
and can be proven along the lines of proof of $(1)\implies (5)$ of Corollary 4.5 in \cite{SY3}; instead of 
Corollary 3.9 and Theorem 4.3 there, one has to use Theorems \ref{distal-g} and \ref{unip}. 
We will give a proof here for the sake of completeness. 

\begin{proof}[Proof of Corollary \ref{inner}] Let $G$ be a connected Lie group. The statement 
in (5) implies that $\Inn(G)$ is compact, and hence it follows that $(5)\implies (4)$. 
Also $(4)\implies (3)\implies (1)$ and $(4)\implies (2)\implies (1)$ are obvious. 
It is enough to show that $(1)\implies (5)$. 

Suppose (1) holds; i.e.\ every inner automorphism of $G$ acts distally on $\spg$. If $G$ is either compact 
or a vector group, then (5) holds trivially. Now suppose $G$ neither compact nor a vector group. 
Since every inner automorphism acts trivially on the center of $G$, we have by Theorem \ref{distal-g} that 
it acts distally on $G$. Therefore, $G$ is distal (cf.\ \cite{Ro}), i.e.\ the action of $\Inn(G)$ on $G$ is distal. 
It follows from Theorem 9 of \cite{Ro} and Corollary 2.1\,(ii) of \cite{Je} that $G/R$ is compact, where $R$ 
is the radical of $G$. Let $N$ be the nilradical of $G$ and let $x\in N$. Then $\inn(x)$ is unipotent, and 
by Theorem \ref{unip}, $\inn(x)=\Id$. This implies that $x$ belongs to the center of $G$ for every $x\in N$. Thus 
$N$ is abelian and central in $G$. Now for the radical $R$, we know that $[R,R]\subset N$ is central in $G$, 
and hence $R$ is nilpotent and it is same as $N$ and it is abelian and central in $G$. Now $G/N$ is compact,  
where $N$ is central in $G$, i.e.\ $N=\R^n\times M$, for some $n\in\N$, where $M$ is compact. Here, 
$V:=\R^n$ is a vector subgroup which is central. By Lemma 3.7 of \cite{I}, $G=\R^n\times K$, where $K$ 
is the maximal compact subgroup of $G$. Thus (5) holds.
\end{proof}

The distal actions of automorphisms $T$ of a connected Lie group $G$ on $\Sub^a_G$ were studied 
earlier if $G$ is in class $\cp$ or $T$ is unipotent. Here we have studied a more general case when 
$T$ acts distally on the maximal central torus of $G$ and the action is considered on $\spg$, a smaller 
subspace of $\Sub^a_G$. More generally, we have also studied the actions of subgroups $\Hc$ of $\Aut(G)$ 
on $\spg$ under a condition that the $\Hc$-action on the maximal central torus of $G$ is distal. This latter 
condition is satisfied if $\Hc$ is contained in $\Aut(G)^0$ or more generally, if it is contained in an almost 
algebraic subgroup of $\Aut(G)$. 

The following example of a connected solvable Lie group $G$, known as the Walnut group, is not covered 
by Theorem \ref{main} or any earlier theorems for the distal action of a general 
automorphism on $\Sub^a_G$, as it has nontrivial central torus and $\Aut(G)$ is not almost algebraic 
(see \cite{DS2}). But it is covered by theorems proven here.

\begin{example}
Let $\mathbb{H}$ be the 3-dimensional Heisenberg group consisting of $\,3\times 3$ real strictly upper triangular matrices. 
There is a canonical action of $SL(2,\R)$ on $\mathbb{H}$ which fixes elements of its center, which is isomorphic to $\R$. 
Let $N=\mathbb{H}/D$, where $D$ is a discrete central subgroup isomorphic to $\Z$. Then $N$ is a connected 2-step 
nilpotent group and the $\SL(2,\R)$-action on $\mathbb{H}$ carries over to the action on $N$. Let $G=SO(2,\R)\ltimes N$, 
the Walnut group. Since the center of $G$ is isomorphic to $\Sc^1$, $\Aut(G)$ acts distally on it. 
Thus Theorems \ref{distal-g} and \ref{general} hold for the action of any $T\in\Aut(G)$ 
and Theorem \ref {subgp-action} holds for any subgroup $\Hc$ of $\Aut(G)$. 
\end{example}

One can also take a non-solvable group $G=SL(2,\R)\ltimes N$, for $N$ mentioned as above, 
whose automorphism group is almost algebraic (see \cite{DS2}). Theorems mentioned in the above example 
hold also for this group. This group $G$ is mentioned in Remark 4.8 of \cite{SY3}, although it has a maximal torus 
of dimension 2, $\Aut(G)$ is almost algebraic (see \cite{DS2}), and the conclusion in Theorem 4.1 of \cite{SY3}
mentioned in the remark holds for this $G$. 

It would be interesting to characterise general automorphisms of any connected Lie group $G$ 
which act distally on $\Sub_G$ or its closed invariant subspace without 
any condition involving the action on the central torus. 

\smallskip
\noindent{\bf Acknowledgement.} The authors would like to acknowledge the support in part by the International 
Centre for Theoretical Sciences (ICTS) during a visit for participating in the program - ICTS Ergodic Theory 
and Dynamical Systems (code: ICTS/etds-2022/12).


\begin{thebibliography}{99}
\bibitem{Ab1} H.\ Abels, `Distal affine transformation groups', {\em J.\ Reine Angew.\ Math.}\ {\bf 299/300} (1978), 294--300.

\bibitem{Ab2} H.\ Abels, `Distal automorphism groups of Lie groups', {\em J.\ Reine Angew.\ Math.}\ {\bf 329} (1981), 82--87.

\bibitem{BC1} H.\ Baik and L.\ Clavier, `The space of geometric limits of one-generator closed subgroups of
${\rm PSL}_2(\R)$', {\em Algebr.\ Geom.\ Topol.} {\bf 13} (2013), 549--576.

\bibitem{BC2} H.\ Baik and L.\ Clavier, `The space of geometric limits of abelian subgroups of PSL$_2(\C)$', 
Hiroshima Math. J. {\bf 46} (2016), 1--36.

\bibitem{Ba} H.\ Bass, `Groups of integral representation type', {\em Pacific J.\ Math.} {\bf 86}
(1980), 15--51.

\bibitem{BP} R.\ Benedetti and C.\ Petronio, {\em Lectures on Hyperbolic Geometry}, Universitext, Springer, Berlin,
1992.

\bibitem{Bh} S.\ Bhattacharya, `Expansiveness of algebraic actions on connected groups', 
{\em Transac.\ Amer.\ Math.\ Soc.} {\bf 356} (2004), 4687--4700.

\bibitem{BHK} M.\ R.\ Bridson, P.\ de la Harpe, and V.\ Kleptsyn, `The Chabauty space of closed subgroups of
the three-dimensional Heisenberg group', {\em Pacific J.\ Math.} {\bf 240} (2009), 1--48.

\bibitem{Ch} C.\ Chevalley, {\em Theory of Lie groups}, Princeton University Press, 1946.

\bibitem{CFY} M.\ Choudhuri, G.\ Faraco and A.\ K.\ Yadav, `On the distality and expansivity of certain maps on spheres',  
{\em Dyn.\ Syst.} {\bf 39} (2024), 166--181.

\bibitem{CG} J.\ P.\ Conze and Y.\ Guivarc'h, `Remarques sur la distalit\'e dans les espaces vectoriels',  
{\em C.\ R.\ Acad.\ Sci.\ Paris Ser.\ A} {\bf 278} (1974), 1083--1086.

\bibitem{D} S.\ G.\ Dani, `On automorphism groups of connected Lie groups', 
{\em Manuscripta Math.}\ {\bf 74} (1992), 445--452.

\bibitem{DM} S.\ G.\ Dani and M.\ McCrudden, `Convolution roots and embeddings of prob-
ability measures on Lie groups', {\em Adv.\ Math.} {\bf 209} (2007), 198 --211.

\bibitem{DS1} S.\ G.\ Dani and R.\ Shah, `Contractible measures and Levy's measures on Lie
groups'. In: Probability on Algebraic Structures (Ed. G.\ Budzban P.\ Feinsilver
and A.\ Mukherjea), {\em Contemporary Math.} {\bf 261} (2000), 3--13.

\bibitem{DS2} S.\ G.\ Dani and R.\ Shah, `On the almost algebraicity of groups of
automorphisms of connected Lie groups', \textcolor{blue}{\url{https://doi.org/10.48550/arXiv.2504.18641}}

\bibitem{E} R.\ Ellis, `Distal transformation groups', {\em Pacific\ J.\ Math.} {\bf 8} (1958), 401--405.

\bibitem{FP} J.\ Fresnel and M.\ van der Put, `Compact subgroups of $\GL(n;\C)$', {\em Rend.\ 
Semin.\ Mat. Univ.\ Padova} {\bf 116} (2006), 187--192.

\bibitem{Fu} H.\ Furstenberg, `The structure of distal flows', {\em Amer.\ J.\ Math.} {\bf 85} (1963), 477--515.

\bibitem{G} T.\ Gelander, `A lecture on invariant random subgroups'. In: {\em New Directions in Locally Compact Groups}, 
pp.\ 186--204, London Math.\ Soc.\ Lecture Note Ser.\ 447, Cambridge Univ. Press, Cambridge, 2018.

\bibitem{HamKa} H.\ Hamrouni and B.\ Kadri, `On the compact space of closed subgroups of locally compact
groups', {\em J.\ Lie Theory} {\bf 24} (2014), 715--723.

\bibitem{HSi} W.\ Hazod and E.\ Siebert, `Automorphisms on a Lie group contracting modulo
a compact subgroup and applications to semistable convolution semigroups', {\em J.\ Theoret.\ Probab.} 
{\bf 1} (1988), 211--225.

\bibitem{Hoc} G.\ Hochschild, `The automorphism group of a Lie group', {\em Trans.\ Amer.\ Math.\ Soc.} 
{\bf 72} (1952), 209--216.

\bibitem{Hoc2} G.\ Hochschild, {\em The Structure of Lie Groups}. Holden-Day, San Francisco,1965.

\bibitem{HN} K.\ H.\ Hofmann and K.-H.\ Neeb, `The compact generation of closed subgroups of locally compact groups', 
{\em J.\ Group Theory} {\bf 12} (2009), 555--559. 

\bibitem{I}  K.\ Iwasawa, `On some types of topological groups', {\em Ann.\ Math.} {\bf 50} (1949), 507--558.

\bibitem{Ja} W.\ Jaworski, `Contraction groups, ergodicity and distal properties of
automorphisms on compact groups', {\em Illinois J.\ Math.} {\bf 56} (2012), 1023--1084.

\bibitem{Je} J.\ Jenkins, `Growth of connected locally compact groups',  {\em J.\ Functional Analysis} {\bf 12} (1973), 113--127.

\bibitem{Kl} B.\ Kloeckner, `The space of closed subgroups of $\R^n$ is stratified and simply connected', {\em J.\ Topology}. 
{\bf 2} (2009), 570--588.

\bibitem{Kn} A.\ W.\ Knapp, {\em Lie Groups Beyond an Introduction}. Second Edition. Progress in Mathematics 140. 
(Birkh\"auser Boston Inc., 2002).  

\bibitem{M9} C.\ C.\ Moore, `Distal affine transformation groups', {\em Amer.\ J.\ Math.} {\bf 90} (1968), 733--751.

\bibitem{PaS} R.\ Palit and R.\ Shah, `Distal actions of automorphisms of nilpotent groups $G$ on Sub$_G$ and 
applications to lattice in Lie groups', {\em Glasgow Math.\ J.} {\bf 63} (2021) 343--362.

\bibitem{PaPrS} R.\ Palit, M.\ B.\ Prajapati and R.\ Shah, `Dynamics of actions of automorphisms of discrete
groups $G$ on Sub$_G$ and applications to lattices in Lie groups', {\em Groups Geom.\ Dyn.} {\bf 17} (2023), 185--213.

\bibitem{PH} I.\ Pourezza and J.\ Hubbard, `The space of closed subgroups of $\R^2$', {\em Topology} {\bf 18} (1979), 143--146.

\bibitem{RS1} C.\ R.\ E.\ Raja and R.\ Shah, `Distal actions and shifted convolution property', 
{\em Israel J.\ Math.} {\bf 177} (2010), 301--318.

\bibitem{RS2} C.\ R.\ E.\ Raja and R.\ Shah, `Some properties of distal actions on locally compact
groups', {\em Ergodic Theory and Dynam. Sys.} {\bf 39} (2019) 1340--1360.

\bibitem{Ro} J.\ Rosenblatt, `A distal property of groups and the growth of connected locally compact groups', 
{\em Mathematika} {\bf 26} (1979), 94--98.

\bibitem{Sh1} R.\ Shah, `Orbits of distal actions on locally compact groups', {\em J.\ Lie Theory} {\bf 22} (2010), 586--599.

\bibitem{Sh2} R.\ Shah, `Expansive automorphisms on connected locally compact groups', {\em New York J.\ Math.} 
{\bf 26} (1920), 285--302.

\bibitem{SY1} R.\ Shah and A.\ K.\ Yadav, `Dynamics of certain distal actions on spheres',
{\em Real Analysis Exchange} {\bf 44} (2019), 77--88.

\bibitem{SY4} R.\ Shah and A.\ K.\ Yadav, `Errata: Dynamics of certain distal actions on spheres', 
{\em Real Analysis Exchange} {\bf 47} (2022), 467-468.

\bibitem{SY2} R.\ Shah and A.\ K.\ Yadav, `Distally of certain actions on $p$-adic spheres', {\em J.\ 
Australian Math. Soc.} {\bf 109} (2020), 250--261.

\bibitem{SY3} R.\ Shah and A.\ K.\ Yadav, `Distal action of Automorphisms of Lie groups $G$ on $\Sub_G$', 
{\em Math.\ Proc.\ Camb.\ Phil.\ Soc.} {\bf 173} (2022), 457--478.

\bibitem{Si} E.\ Siebert, `Contractive automorphisms on locally compact groups', {\em Math.\ Zeit.} {\bf 191} (1986), 
73--90.

\bibitem{St06} M.\ Stroppel, {\em Locally Compact Groups}. EMS Textbooks in Mathematics (European Math.\ 
Soc.\ - EMS, 2006).

\bibitem{W} S.\ P.\ Wang, `On the limit of subgroups in a group', {\em Amer.\ J.\ Math.} {\bf 92} (1970),
708--724.

\end{thebibliography}
\end{document}